\documentclass[psamsfonts]{amsart}
\usepackage{amssymb}
\usepackage{graphicx,color}
\usepackage{amssymb,amsfonts,amsthm}
\usepackage[all,arc]{xy}
\usepackage{enumerate}
\usepackage{mathrsfs}

\newtheorem{thm}{Theorem}[section]

\newtheorem{lemma}[thm]{Lemma}
\newtheorem{prop}[thm]{Proposition}

\theoremstyle{definition}
\newtheorem{defn}[thm]{Definition}

\theoremstyle{remark}

\makeatletter
\let\c@equation\c@thm
\makeatother
\numberwithin{equation}{section}

\title[Modified Gursky-Streets equation]{Existence of solution to modified Gursky-Streets equation}

\thanks{The second author is supported by
National Natural Science Foundation of China (No. 12001138).}

\author{Yi Huang}

\author{Zhenan Sui}

\author{Mingyu Xie}

\address{School of Mathematics, Harbin Institute of Technology, Harbin, China}
\email{24s012003@stu.hit.edu.cn}

\address{Institute for Advanced Study in Mathematics of HIT, Harbin Institute of Technology, Harbin, China}
\email{sui.4@osu.edu}

\address{School of Mathematics, Harbin Institute of Technology, Harbin, China}
\email{24s012009@stu.hit.edu.cn}

\begin{document}

\begin{abstract}
We solve the modified Gursky-Streets equation, which is a fully nonlinear equation arising in conformal geometry with uniform $C^{1, 1}$ estimates when (i) $\gamma > 0$ and  $1 \leq k \leq n$ or (ii) $r > 0$ and $2 s k \leq r n$. We also prove the existence of a Lipschitz continuous viscosity solution when $r \neq 0$.
\end{abstract}

\subjclass[2010]{Primary 53C21; Secondary 35J60}

\maketitle

%\tableofcontents

\section {\large Introduction}

\vspace{4mm}

On a smooth compact Riemannian manifold $(M^n, g)$ of dimension $n \geq 3$, we are interested in solving the following class of conformal curvature equations
\begin{equation} \label{eq1}
u_{tt} \sigma_k \big( W[u] \big) -  \sigma_{k}^{ij} \big( W[u] \big)  u_{ti} u_{tj} = \psi(x, t) \quad \text{on} \quad M \times [0, 1]
\end{equation}
subject to the boundary condition
\begin{equation} \label{eq1-4}
u(\cdot, 0) = u_0, \quad u(\cdot, 1) = u_1,
\end{equation}
where
\begin{equation} \label{eq1-7}
W[u] = g^{- 1} \bigg( \nabla^2 u + s d u \otimes d u + \Big( \gamma \Delta u - \frac{r}{2} |\nabla u|^2 \Big) g + A \bigg),
\end{equation}
$\gamma, s, r \in \mathbb{R}$, $\gamma \geq 0$, $A$ is a smooth symmetric tensor on $M$ with $\lambda ( g^{- 1} A ) \in \Gamma_k$, and $\psi \geq 0$ is a given smooth function defined on $M \times [0, 1]$, $u_0$, $u_1$ are given smooth functions on $M$ with $\lambda\big( W[u_0] \big) \in \Gamma_k$ and $\lambda\big( W[u_1] \big) \in \Gamma_k$. Also,  $\nabla{u}$, $\nabla^{2}u$ and $\Delta u$ are the gradient, Hessian and Laplace-Beltrami operator of $u$ with respect to the background metric $g$ respectively.
In order to make the notation and computation easier, we always choose a smooth local orthonormal frame field $e_1, \ldots, e_n$ on $M$ with respect to the metric $g$. Thus, $u_{ti} := \nabla_{e_i} u_t$, $u_{ij} := \nabla_{e_j e_i} u$, and higher order covariant derivatives can be similarly written in this manner.
Besides,
\[  \sigma_k \big( W[u] \big) := \sigma_k \Big( \lambda \big( W[u] \big) \Big), \quad \sigma_{k}^{ij} \big( W[u] \big) : = \frac{\partial \sigma_k \big( W[u] \big)}{\partial W_{ij}[u]},  \]
where $\lambda \big( W[u] \big) = ( \lambda_1, \ldots,
\lambda_n )$ are the eigenvalues of the matrix $W[u]$, and
\[ \sigma_k (\lambda) =  \sum\limits_{ 1 \leq i_1 <
\cdots < i_k \leq n} \lambda_{i_1} \cdots \lambda_{i_k} \]
is the $k$th elementary symmetric function defined on
Garding's cone
\[\Gamma_k = \{ \lambda  \in \mathbb{R}^n : \sigma_j ( \lambda ) > 0, \, j = 1,
\ldots, k \}. \]

If we set
\[ E_u = u_{tt} W[u] - g^{- 1} d u_t \otimes d u_t, \]
we arrive at the following relation
\begin{equation} \label{eq4}
u_{tt} \sigma_k \big( W[u] \big) -  \sigma_{k}^{ij} \big( W[u] \big)  u_{ti} u_{tj} = u_{tt}^{1 - k} \sigma_k ( E_u )
\end{equation}
by Proposition \ref{prop4}. Thus equation \eqref{eq1} is equivalent to
\begin{equation} \label{eq2}
u_{tt}^{1 - k} \sigma_k (E_u) = \psi(x, t).
\end{equation}

When $\psi > 0$ throughout $M \times [0, 1]$, following \cite{HeXuZhang} we call a $C^2$ function $u$ on $M \times [0, 1]$ admissible if
\[ \lambda \big( W [u] \big) \in \Gamma_k, \quad  u_{tt} > 0, \quad \sigma_k (E_u) > 0.  \]
Consequently, equation \eqref{eq1} or \eqref{eq2} is elliptic for $C^2$ solution $u$ with $\lambda \big( W [u] \big) (x, t) \in \Gamma_k$ for any $(x, t) \in M \times [0, 1]$ (see Proposition \ref{Prop7}).

When equation \eqref{eq1} becomes degenerate, we call a $C^{1, 1}$ function $u$ on $M \times [0, 1]$ admissible if
\[ \lambda \big( W [u] \big) \in \overline{\Gamma}_k, \quad  u_{tt} \geq 0, \quad \sigma_k (E_u) \geq 0  \]
almost everywhere.

If $\psi \equiv 0$, $s = 1$, $r = 1$, $\gamma = 0$, and $A$ is the Schouten tensor of $g$, then we obtain the following degenerate conformal curvature equation
\begin{equation} \label{eq1-2}
u_{tt} \sigma_k \big( A_u \big) -  \sigma_{k}^{ij} \big( A_u \big)  u_{ti} u_{tj} =  0
\end{equation}
with
\begin{equation} \label{eq1-3}
A_u =  \nabla^2 u + d u \otimes d u  - \frac{1}{2} |\nabla u|^2 g + A.
\end{equation}
Equation \eqref{eq1-2} was introduced by Gursky and Streets \cite{Gursky-Streets2018}. In this important work, the uniqueness of the $\sigma_2$-Yamabe equation in dimension $n = 4$ is proved. A key ingredient of their proof involves a study of equation \eqref{eq1-2}. This equation is now known as the Gursky-Streets equation, which turns out to be the geodesic equation under a Riemannian metric defined on a conformal class of metrics. Subsequently, the Gursky-Streets equation was studied by He \cite{He2021} for $k = 2$ in dimension $n \geq 4$, where he established uniform $C^{1, 1}$ regularity for $k = 2$ and reproved the uniqueness of the $\sigma_2$ equation in dimension $n = 4$. Then, the case when $k \leq \frac{n}{2}$ was investigated by He, Xu and Zhang \cite{HeXuZhang} by a crucial proof of the concavity of the operator associated to the equation for all $1 \leq k \leq n$ (see Proposition \ref{prop5}), which relies on Garding theory of hyperbolic polynomials and results in the theory of real roots for interlacing polynomials.
It is worth mentioning that the geometry of Gursky-Streets' metric on the space of conformal metrics has a parallel theory with the geometry of the space of K\"ahler metrics, where the geodesic equation can be written as a homogeneous complex Monge-Amp\`ere equation (see \cite{Calabi1982, Mabuchi1987, Mabuchi1986, Semmes1992, Donaldson1999, ChenXX2000, Donaldson2004}).

Motivated by the above work, we are especially interested in the existence and regularity of solutions to the generalized Gursky-Streets equation \eqref{eq1}.
Under conformal deformation of metric $g_u = e^{2 u} g$, the Ricci tensor $\text{Ric}_u$ of $g_u$ and $\text{Ric}$ of $g$ are related by the formula
\begin{equation} \label{eq1-5}
- \frac{\text{Ric}_u}{n - 2} = \nabla^2 u - d u \otimes d u + \Big( \frac{\Delta u}{n - 2} + |\nabla u|^2 \Big) g - \frac{\text{Ric}}{n - 2}.
\end{equation}
Moreover, the Schouten tensor
\[ A_u = \frac{1}{n - 2} \Big( Ric_u - \frac{S_u}{2 (n - 1)}  g_u \Big) \]
of $g_u$ and $A$ of $g$ are related by the formula
\begin{equation} \label{eq1-6}
- A_u = \nabla^2 u - d u \otimes d u + \frac{1}{2} |\nabla u|^2 g - A ,
\end{equation}
where $S_u$ is the scalar curvature of $g_u$. If we change $u$ into $- u$ in \eqref{eq1-6}, we will obtain exactly \eqref{eq1-3}.
In addition, in Gursky and Viaclovski \cite{GV03} and Li and Li \cite{LL2005}, the following modified Schouten tensor with a parameter $\tau$
\[  A_u^{\tau} = \frac{1}{n - 2} \Big( Ric_u - \frac{\tau S_u}{2 (n - 1)}  g_u \Big)  \]
was introduced, which satisfies
\begin{equation} \label{eq1-10}
- A_u^{\tau} = \nabla^2 u - d u \otimes d u + \bigg( \frac{1 - \tau}{n - 2} \Delta u + \Big( 1 - \frac{\tau}{2} \Big) |\nabla u|^2 \bigg) g - A^{\tau}.
\end{equation}

All \eqref{eq1-3}, \eqref{eq1-6}, \eqref{eq1-5} and \eqref{eq1-10} with $\tau \leq 1$ fall into the form of \eqref{eq1-7}, which are known as conformal Schouten tensor, conformal negative Schouten tensor, conformal negative Ricci tensor and conformal negative modified Schouten tensor respectively.

A central problem in conformal geometry is on the solvability of fully nonlinear Yamabe type problem which can be written in the general form
\begin{equation} \label{eq1-8}
\sigma_k \big( W [u] \big) = e^{- 2 k u} \psi(x),
\end{equation}
or
\begin{equation} \label{eq1-9}
\sigma_k \big( W [u] \big) = e^{2 k u} \psi(x),
\end{equation}
where $\psi(x)$ is a given smooth positive function defined on smooth compact manifold with or without boundary. There is a vast amount of literature on the existence and regularity of solutions to \eqref{eq1-8} or \eqref{eq1-9} ever since the work of Viaclovski \cite{Viaclovski2000}. In particular, we notice that the tensor \eqref{eq1-3}, \eqref{eq1-6} and \eqref{eq1-5} may bring very different existence and regularity issues to equations in the form \eqref{eq1-8} or \eqref{eq1-9}. The estimates in \cite{Viaclovsky2002} and \cite{GV03} reflect such difference (see also \cite{CGH22} for a more general equation). Moreover, when solving fully nonlinear Loewner-Nirenberg type problems, we observe that the negative Schouten tensor case has Lipschitz continuous solution in general, as proved by Gonz\'alez, Li and Nguyen \cite{Gonzalez-Li-Nguyen}, and there are counterexamples to $C^1$ regularity given in Li and Nguyen \cite{LN21} and in Li, Nguyen and Xiong \cite{LNX23}. In contrast, when $\gamma > 0$ (for example, \eqref{eq1-5}), Guan \cite{Guan08, Guan09} obtained smooth solution to \eqref{eq1-9} on smooth compact manifold with boundary, based on which, smooth solution to fully nonlinear Loewner-Nirenberg problem of negative Ricci tensor was obtained in \cite{Sui2024}.  In view of these work, we are interested in whether the $\Delta u$ term can still bring in a priori estimates up to second order for equation \eqref{eq1}--\eqref{eq1-4} when $\psi > 0$ throughout $M \times [0, 1]$. The following result gives an affirmative answer.

\begin{thm} \label{Theorem1}
Let $(M, g)$ be a compact Riemannian manifold with $\lambda(A) \in \Gamma_k$ and suppose that (i) $\gamma > 0$ and  $1 \leq k \leq n$ or (ii) $r > 0$ and $2 s k \leq r n$. Given a smooth positive function $\psi(x, t)$ defined on $M \times [0, 1]$ and smooth function $u_0$, $u_1$ defined on $M$ with $\lambda\big( W[u_0] \big) \in \Gamma_k$ and $\lambda\big( W[u_1] \big) \in \Gamma_k$, there exists a unique smooth solution $u$ of \eqref{eq1}--\eqref{eq1-4} on $M \times [0, 1]$ with $\lambda\big( W[ u ] \big) (x, t) \in \Gamma_k$ for any $(x, t) \in M \times [0, 1]$. Moreover, we have the following estimate which is independent of $\inf\psi$,
\[ \max\limits_{M \times [0, 1]} \Big( |u| + |u_t| + |\nabla u| + u_{tt} + |\nabla^2 u| + |\nabla u_t| \Big) \leq C. \]
When $r > 0$ and $2 s k \leq r n$, the $C^2$ estimate is uniform as $\gamma \rightarrow 0^+$.
\end{thm}

Theorem \ref{Theorem1} was proved by He, Xu and Zhang \cite{HeXuZhang} for $W[u] = A_u$ assuming $k \leq \frac{n}{2}$. The condition $k \leq \frac{n}{2}$ plays a crucial role in their derivation of $C^2$ estimates due to the pattern $d u \otimes d u  - \frac{1}{2} |\nabla u|^2 g$ in \eqref{eq1-3}.  We extend their result to the case when $r > 0$ and $2 s k \leq r n$. In contrast, if we assume $\gamma > 0$, the $\Delta u$ term can bring in good terms for $C^2$ estimates of \eqref{eq1}--\eqref{eq1-4}, regardless of the pattern of gradient terms. Our test functions may be different for different situations and accordingly the estimates are different.

In order to apply continuity method to prove existence of admissible solution to \eqref{eq1}--\eqref{eq1-4}, we need to find an admissible function which can serve as an initial solution of the continuity process. Different from \cite{HeXuZhang}, where
\[ \ln \Big( (1 - t) e^{u_0} + t e^{u_1} \Big) + a t (t - 1) \]
with sufficiently large $a$ is admissible due to the pattern $d u \otimes d u  - \frac{1}{2} |\nabla u|^2 g$, we establish the following existence and regularity result for $\gamma > 0$ which seems to be useful in itself.

\begin{thm} \label{Thm4}
Let $(M, g)$ be a compact Riemannian manifold with $\lambda(A) \in \Gamma_k$ and suppose that $\gamma > 0$. Given a smooth positive function $\psi(x, t)$ defined on $M \times [0, 1]$ and smooth function $u_0$, $u_1$ defined on $M$ with $\lambda\big( W[u_0] \big) \in \Gamma_k$ and $\lambda\big( W[u_1] \big) \in \Gamma_k$, which also satisfy the compatibility condition:
\begin{equation} \label{eq5-2}
e^{- 2 k u_0} \sigma_k \big( W[u_0] \big) = \psi(x, 0) \quad \text{and} \quad e^{- 2 k u_1} \sigma_k \big( W[u_1] \big) = \psi(x, 1),
\end{equation}
there exists a unique smooth solution $u(x, t)$ to
\begin{equation} \label{eq5-1}
\left\{ \begin{aligned}
& e^{- 2 k u} \sigma_k \big( W[u] \big) = \psi(x, t) \quad \text{on} \quad M \times [0, 1], \\
& u(\cdot, 0) = u_0, \quad  u(\cdot, 1) = u_1
\end{aligned} \right.
\end{equation}
with $\lambda\big( W[ u ] \big) (x, t) \in \Gamma_k$ for any $( x, t ) \in M \times [0, 1]$. Moreover, we have the estimate
\[ \max\limits_{M \times [0, 1]} \Big( |u| + |u_t| + |\nabla u| + |\nabla^2 u| \Big) \leq C. \]
\end{thm}

The proof of Theorem \ref{Thm4} exploits the technique in Schauder theory, which also works for other type of fully nonlinear elliptic equations.

Based on Theorem \ref{Theorem1}, we next study the degenerate equations.
First, we observe that
if we merely assume $\gamma \geq 0$ but $r \neq 0$, we can still obtain $C^1$ estimate for \eqref{eq1}--\eqref{eq1-4} which is independent of $\inf\psi$ and $\gamma \geq 0$:
\begin{equation} \label{eq1-11}
\max\limits_{M \times [0, 1]} \Big( |u| + |u_t| + |\nabla u| \Big) \leq C.
\end{equation}
Combining \eqref{eq1-11} and Theorem \ref{Theorem1}, we obtain the existence of a Lipschitz continuous viscosity solution (see Definition \ref{Def1}) to \eqref{eq1}--\eqref{eq1-4}.

\begin{thm} \label{Thm-weak solution}
Let $(M, g)$ be a compact Riemannian manifold with $\lambda(A) \in \Gamma_k$ and suppose that $r \neq 0$. Given a smooth nonnegative function $\psi(x, t)$ defined on $M \times [0, 1]$ and smooth function $u_0$, $u_1$ defined on $M$ with $\lambda\big( W[u_0] \big) \in \Gamma_k$ and $\lambda\big( W[u_1] \big) \in \Gamma_k$, there exists a Lipschitz continuous viscosity solution $u$ of \eqref{eq1}--\eqref{eq1-4} on $M \times [0, 1]$.
\end{thm}

Since the estimates in Theorem \ref{Theorem1} are independent of $\inf \psi$, we obtain the existence of $C^{1, 1}$ solution to the degenerate equation \eqref{eq1}--\eqref{eq1-4} when (i) $\gamma > 0$ and  $1 \leq k \leq n$ or (ii) $r > 0$ and $2 s k \leq r n$. In particular, we obtain the following existence result.

\begin{thm} \label{Theorem2}
Let $(M, g)$ be a compact Riemannian manifold with $\lambda(A) \in \Gamma_k$ and suppose that (i) $\gamma > 0$ and  $1 \leq k \leq n$ or (ii) $r > 0$ and $2 s k \leq r n$. Given smooth function $u_0$, $u_1$ defined on $M$ with $\lambda\big( W[u_0] \big) \in \Gamma_k$ and $\lambda\big( W[u_1] \big) \in \Gamma_k$, there exists an admissible solution $u \in C^{1, 1} \big( M \times [0, 1] \big)$ which satisfies
\begin{equation} \label{eq1-1}
\left\{
\begin{aligned}
& u_{tt} \sigma_k \big( W[u] \big) - \sigma_{k}^{ij} \big( W[u] \big)  u_{ti} u_{tj} = 0, \\
& u(\cdot, 0) = u_0, \quad  u(\cdot, 1) = u_1
\end{aligned}
\right.
\end{equation}
almost everywhere on $M \times [0, 1]$. Moreover, we have the following estimate
\[ \sup \limits_{M \times [0, 1]} \Big( |u| + |u_t| + |\nabla u| + u_{tt} + |\nabla^2 u| + |\nabla u_t| \Big) \leq C. \]
For  (i) $\gamma > 0$,  $r \geq 0$ and $2 s k \leq r n$ or (ii) $r > 0$ and $2 s k \leq r n$, the $C^{1, 1}$ admissible solution is unique.
\end{thm}
We note that $u$ in Theorem \ref{Theorem2} is a viscosity solution. The existence part of Theorem \ref{Theorem2} was proved in \cite{HeXuZhang} for $W[u] = A_u$ assuming $k \leq \frac{n}{2}$, under which  uniqueness result was also obtained because of the pattern $d u \otimes d u  - \frac{1}{2} |\nabla u|^2 g$ in \eqref{eq1-3} which can lead to an approximation result. Our uniqueness result exploits the same approximation method, while existence part enriches previous result.

This paper is organized as follows. The necessary preliminaries and estimates for $u$ and $u_t$ are presented in section 2. Section 3 is on estimate of $|\nabla u|$, while section 4 and 5 are devoted to second order estimates. The existence is proved in section 6 and uniqueness is given in section 7.

\medskip

\vspace{4mm}

\section{Preliminary estimates}

\vspace{4mm}

In this section, we first provide some basic facts.
Let $M_{n + 1}$ be the set of real symmetric $(n + 1) \times (n + 1)$ matrices. For $R \in M_{n + 1}$, we write
\[ R = (r_{IJ})_{0 \leq I, J \leq n}, \]
and denote the submatrix $(r_{ij})_{1 \leq i, j \leq n}$ by $r$.
We have the following fact on concavity, which was proved in He, Xu and Zhang \cite{HeXuZhang}.
\begin{prop} \label{prop5}
The set
\[ \mathcal{S} = \{ R \in M_{n + 1} | \lambda(r) \in \Gamma_k, \, F_k (R):= r_{00} \sigma_k (r) - \sigma_k^{ij} (r) r_{0i} r_{0j} > 0 \} \]
is a convex cone. Besides,
$F_k^{\frac{1}{k + 1}}(R)$, and hence $\ln F_k (R)$ are concave on $\mathcal{S}$ for $1 \leq k \leq n$.
\end{prop}

We also need the following fact which can be found in Gursky and Streets \cite{Gursky-Streets2018}.
\begin{prop}  \label{prop4}
Given a symmetric $n \times n$ matrix $A$ and a vector $X = (X_1, \cdots, X_n)$, we have
\[ \sigma_k^{ij} (A - X \otimes X) X_i X_j =  \sigma_k^{ij} (A) X_i X_j , \]
\[ \sigma_k (A - X \otimes X) = \sigma_k (A) - \sigma_k^{ij}(A) X_i X_j. \]
\end{prop}

The following proposition is proved in \cite{HeXuZhang}.
\begin{prop} \label{prop6}
Let $u$ be a $C^2$ function. Suppose that $\lambda \big( W [ u ] \big) \in \Gamma_k$, $u_{tt} > 0$ and $\sigma_k (E_u) > 0$, then $\lambda(E_u) \in \Gamma_k$.
\end{prop}

By direct computation, the linearized operator associated to equation \eqref{eq2} is
\[  \begin{aligned}
& \mathcal{L} (v) := (1 - k) u_{tt}^{- k} \sigma_k (E_u) v_{tt} \\
& + u_{tt}^{1 - k} \sigma_{k}^{ij} (E_u) \Big( v_{tt} W_{ij}[u] + u_{tt} \mathcal{M}_{ij} (v) - u_{t i}  v_{tj} - v_{ti} u_{tj} \Big),
\end{aligned} \]
where
\[ \mathcal{M} (v) := g^{-1} \Big( \nabla^2 v + s d u \otimes d v + s d v \otimes d u + \big( \gamma \Delta v - r \langle \nabla u, \nabla v \rangle \big) g  \Big). \]

\begin{prop} \label{Prop7}
When $\psi > 0$ throughout $M \times [0, 1]$, equation \eqref{eq1} (or equivalently, \eqref{eq2}) is elliptic for $C^2$ solution $u$ with $\lambda\big( W[u] \big) \in \Gamma_k$.
\end{prop}
\begin{proof}
It suffices to consider the operator
\[  \begin{aligned}
& \tilde{\mathcal{L}} (v) = (1 - k) u_{tt}^{- 1} \sigma_k (E_u) v_{tt} \\
& + \sigma_{k}^{ij} (E_u) \Big( v_{tt} W_{ij}[u] + u_{tt} \big( v_{ij} + \gamma \Delta v \delta_{ij} \big) - u_{t i}  v_{tj} - v_{ti} u_{tj} \Big).
\end{aligned} \]
By direct computation, we have
\[ \begin{aligned}
\frac{\partial\tilde{\mathcal{L}}(v)}{\partial v_{tt}} = & (1 - k) u_{tt}^{- 1} \sigma_k (E_u) + \sigma_{k}^{ij} (E_u) W_{ij}[u], \\
\frac{\partial\tilde{\mathcal{L}}(v)}{\partial v_{ti}} = & - \sigma_{k}^{ij} (E_u)  u_{tj}, \\
\frac{\partial\tilde{\mathcal{L}}(v)}{\partial v_{tj}} = & - \sigma_{k}^{ij} (E_u)  u_{ti},  \\
\frac{\partial\tilde{\mathcal{L}}(v)}{\partial v_{ij}} = & \Big( \sigma_{k}^{ij} (E_u) + (n - k + 1) \sigma_{k - 1}(E_u) \gamma \delta_{ij} \Big) u_{tt}.
\end{aligned} \]
For any $(\xi_0, \xi_1, \cdots, \xi_n) \in \mathbb{R}^{n + 1}$,
\begin{equation*}
\begin{aligned}
& \Big( (1 - k) u_{tt}^{- 1} \sigma_k (E_u) + \sigma_{k}^{ij} (E_u) W_{ij}[u] \Big) \xi_0^2 - \sigma_{k}^{ij} (E_u)  u_{tj} \xi_0 \xi_i  - \sigma_{k}^{ij} (E_u)  u_{ti} \xi_0 \xi_j \\
& +  \Big( \sigma_{k}^{ij} (E_u) + (n - k + 1) \sigma_{k - 1}(E_u) \gamma \delta_{ij} \Big) u_{tt} \xi_i \xi_j \\
= &  \bigg( (1 - k) u_{tt}^{- 1} \sigma_k (E_u) + \sigma_{k}^{ij} (E_u) \frac{(E_u)_{ij} + u_{ti} u_{tj}}{u_{tt}} \bigg) \xi_0^2 - \sigma_{k}^{ij} (E_u)  u_{tj} \xi_0 \xi_i  \\
& - \sigma_{k}^{ij} (E_u)  u_{ti} \xi_0 \xi_j  +  \Big( \sigma_{k}^{ij} (E_u) + (n - k + 1) \sigma_{k - 1}(E_u) \gamma \delta_{ij} \Big) u_{tt} \xi_i \xi_j \\
= &  \Big( (1 - k) u_{tt}^{- 1} \sigma_k (E_u) + u_{tt}^{- 1} \sigma_{k}^{ij} (E_u) (E_u)_{ij} \Big) \xi_0^2 \\ & + \sigma_{k}^{ij} (E_u) u_{ti} u_{tj} u_{tt}^{- 1} \xi_0^2 - \sigma_{k}^{ij} (E_u)  u_{tj} \xi_0 \xi_i
- \sigma_{k}^{ij} (E_u)  u_{ti} \xi_0 \xi_j  + \sigma_{k}^{ij} (E_u) u_{tt} \xi_i \xi_j \\
& +  (n - k + 1) \sigma_{k - 1}(E_u) \gamma  u_{tt} \sum\limits_{i = 1}^n \xi_i^2 \\
= &   u_{tt}^{- 1} \sigma_k (E_u) \xi_0^2 + (n - k + 1) \sigma_{k - 1}(E_u) \gamma  u_{tt} \sum\limits_{i = 1}^n \xi_i^2 \\
& + \sigma_{k}^{ij} (E_u) \bigg( \frac{u_{ti}}{\sqrt{u_{tt}}} \xi_0 - \sqrt{u_{tt}} \xi_i \bigg) \bigg( \frac{u_{tj}}{\sqrt{u_{tt}}} \xi_0 - \sqrt{u_{tt}} \xi_j \bigg).
\end{aligned}
\end{equation*}
By Proposition \ref{prop6}, we can see the ellipticity of \eqref{eq1} or \eqref{eq2}.
\end{proof}

Next we give some preliminary estimates for admissible solutions to \eqref{eq1}--\eqref{eq1-4}.

\vspace{2mm}

\subsection{$C^0$ estimate}~

\vspace{2mm}

In this subsection, we derive the $C^0$ estimate, following the idea of Gursky-Streets \cite{Gursky-Streets2018}.

\begin{prop} \label{Prop1}
Let $u$ be an admissible solution to \eqref{eq1}--\eqref{eq1-4}. Then
\[ \max\limits_{M \times [0, 1]} |u| \leq C, \]
where $C$ is a positive constant depending only on $|u_0|_{C^0(M)}$, $|u_1|_{C^0(M)}$ and the upper bound of $\psi$.
\end{prop}

\begin{proof}
First, since $u_{tt} > 0$, we have
\[ \frac{u(\cdot, t) - u(\cdot, 0)}{t - 0} < \frac{u(\cdot, 1) - u(\cdot, t)}{1 - t}, \quad \forall t \in (0, 1),  \]
which yields the upper bound
\[ u(\cdot, t) < (1 - t) u(\cdot, 0) + t u(\cdot, 1) = (1 - t) u_0 + t u_1 \quad \forall t \in (0, 1). \]

To find a lower bound of $u$, we consider
\[ \Psi = u + a t (1 - t), \quad t \in [0, 1], \]
where $a$ is a positive constant to be determined.
Assume that $\Psi$ attains an interior minimum at $(x_0, t_0)$. Thus, at $(x_0, t_0)$, we have
\[ \nabla u = 0, \quad \nabla^2 u \geq 0.  \]
In addition,
\begin{equation} \label{eq2-3}
\begin{aligned}
 \mathcal{L} (\Psi) = & (1 - k) u_{tt}^{- k} \sigma_k (E_u) \Psi_{tt} \\
& + u_{tt}^{1 - k} \sigma_{k}^{ij} (E_u) \Big( \Psi_{tt} W_{ij}[u] + u_{tt} \mathcal{M}_{ij} (\Psi) - u_{t i}  \Psi_{tj} - \Psi_{ti} u_{tj} \Big) \\
= & (1 - k) u_{tt}^{- k} \sigma_k (E_u) (u_{tt} - 2 a) \\
& + u_{tt}^{1 - k} \sigma_{k}^{ij} (E_u) \Big( (u_{tt} - 2 a) W_{ij}[u] + u_{tt} \mathcal{M}_{ij} (u) - 2 u_{ti} u_{tj} \Big) \\
= & (1 + k) u_{tt}^{1 - k} \sigma_k (E_u) - 2 a u_{tt}^{- k} \sigma_k (E_u) - u_{tt}^{2 - k} \sigma_k^{ij} (E_u) A_{ij} \\
& - 2 a u_{tt}^{- k} \sigma_k^{ij} (E_u) u_{t i} u_{t j},
\end{aligned}
\end{equation}
where
\[ W_{ij}[u] = \nabla_{ij} u + \gamma \Delta u  \delta_{ij} + A_{ij} , \quad  (E_u)_{ij} = u_{tt} W_{ij}[u] - u_{t i} u_{t j}, \]
\[ \mathcal{M}_{ij} (u) =  \nabla_{ij} u + \gamma \Delta u \delta_{ij}. \]

Since $\Psi$ attains an interior minimum at $(x_0, t_0)$, we know that at $(x_0, t_0)$,
\[ \Psi_{tt} \nabla^2 \Psi - d \Psi_t \otimes d \Psi_t \geq 0. \]
One may see \cite{Gursky-Streets2018} Page 3528--3529 for a proof.
The above inequality means that
\[ (u_{tt} - 2 a ) \nabla^2 u - d u_t \otimes d u_t \geq 0.  \]
It follows that
\[  u_{tt} \nabla^2 u - d u_t \otimes d u_t \geq 0.  \]
Consequently,
\[ E_u = u_{tt} \nabla^2 u + u_{tt} \gamma \Delta u g + u_{tt} A - d u_t \otimes d u_t \geq u_{tt} A, \]
if we assume that $\gamma \geq 0$.
Since we have assumed that $\lambda(A) \in \Gamma_k$, by the concavity of $\sigma_k^{\frac{1}{k}}$ we have
\begin{equation} \label{eq2-1}
 \sigma_k^{ij} (E_u) A_{ij} \geq k \sigma_k^{1 - \frac{1}{k}} (E_u) \sigma_k^{\frac{1}{k}} (A) > 0.
\end{equation}

Also, by Proposition \ref{prop4}, we have
\begin{equation} \label{eq2-2}
\begin{aligned}
 & u_{tt}^{- k} \sigma_k (E_u) +  u_{tt}^{- k} \sigma_k^{ij} (E_u) u_{t i} u_{t j}  \\
= &  u_{tt}^{- k} \sigma_k \big( u_{tt} W[u] -  d u_t \otimes d u_t \big) +  u_{tt}^{- k} \sigma_k^{ij} \big( u_{tt} W[u] - d u_t \otimes d u_t \big) u_{t i} u_{t j} \\
= &  u_{tt}^{- k} \Big( \sigma_k \big( u_{tt} W[u] \big) - \sigma_{k}^{ij} \big( u_{tt} W[u] \big) u_{ti} u_{tj} \Big) +  u_{tt}^{- k} \sigma_k^{ij} \big( u_{tt} W[u] \big) u_{t i} u_{t j} \\
= & \sigma_k \big( W[u] \big).
\end{aligned}
\end{equation}

Taking \eqref{eq2-1} and \eqref{eq2-2} into \eqref{eq2-3}, we can see that at $(x_0, t_0)$,
\[ \begin{aligned}
\mathcal{L} (\Psi) < & (1 + k) u_{tt}^{1 - k} \sigma_k ( E_u ) - 2 a \sigma_k \big( W[u] \big) \\
\leq &  (1 + k) \psi (x, t) - 2 a \sigma_k ( A ).
\end{aligned} \]
By taking $a$ sufficiently large, depending on the upper bound of $\psi$, we can make
\[ \mathcal{L} (\Psi) < 0 \quad \text{  at  } (x_0, t_0). \]
But this is impossible. Hence $\Psi$ can not have an interior minimum. We therefore obtain that
\[ u \geq \min\Big\{ \min\limits_{M}{u_0}, \min\limits_M {u_1} \Big\} - \frac{a}{4}. \]
\end{proof}

\vspace{2mm}

\subsection{Estimate for $u_t$}~

\vspace{2mm}

We shall adopt the idea of Gursky-Streets \cite{Gursky-Streets2018} to derive the bound of $u_t$.

\begin{prop} \label{Prop2}
Let $u$ be an admissible solution to \eqref{eq1}--\eqref{eq1-4}. Then we have
\[ \max\limits_{M \times [0, 1]} |u_t| \leq C, \]
where $C$ is a positive constant depending only on the upper bound of $\psi$, the bound of $|u_0 - u_1|$,  and the positive lower bound of $\sigma_k \big( W[u_0] \big)$, $\sigma_k \big( W[u_1] \big)$.
\end{prop}

\begin{proof}
Since $u_{tt} > 0$, we have
\[ u_t(\cdot, 0) \leq u_t (\cdot, t) \leq u_t (\cdot, 1), \quad \forall \,  t \in [0, 1]. \]
Hence it suffices to give a lower bound for $u_t(\cdot, 0)$ and an upper bound for $u_t (\cdot, 1)$.

First, we give a lower bound for $u_t(\cdot, 0)$. Consider the test function
\[ \Phi (x, t) = u(x, t) - u_0 (x) - b t^2 + c t,  \]
where $b$ and $c$ are positive constants to be determined.

Assume that $\Phi$ attains an interior minimum at $(x_0, t_0) \in M \times (0, 1)$. Then, at $(x_0, t_0)$, we have
\[ \nabla u = \nabla u_0, \quad \nabla^2 u \geq \nabla^2 u_0,  \]
\begin{equation} \label{eq2-4}
\begin{aligned}
& \mathcal{L} (\Phi)
=  (1 - k) u_{tt}^{- k} \sigma_k (E_u) (u_{tt} - 2 b) \\
& + u_{tt}^{1 - k} \sigma_{k}^{ij} (E_u) \Big( (u_{tt} - 2 b) W_{ij}[u] + u_{tt} \big( W_{ij} [u] - W_{ij} [u_0] \big) - 2 u_{ti} u_{tj} \Big) \\
= & (1 + k) u_{tt}^{1 - k} \sigma_k (E_u) - 2 b (1 - k) u_{tt}^{- k} \sigma_k (E_u)
 - 2 b u_{tt}^{1 - k} \sigma_k^{ij} (E_u) W_{ij} [u] \\
& - u_{tt}^{2 - k} \sigma_k^{ij} (E_u) W_{ij}[u_0] \\
= &  (1 + k) u_{tt}^{1 - k} \sigma_k (E_u) - 2 b \sigma_k ( W[u] ) - u_{tt}^{2 - k} \sigma_k^{ij} (E_u) W_{ij}[u_0] ,
\end{aligned}
\end{equation}
where
\[  (E_u)_{ij} = u_{tt} W_{ij}[u] - u_{t i} u_{t j},\]
and the last inequality is derived in the same way as Proposition \ref{Prop1}.

Since $\Phi$ attains an interior minimum at $(x_0, t_0)$, we know that at $(x_0, t_0)$,
\[ \Phi_{tt} \nabla^2 \Phi - d \Phi_t \otimes d \Phi_t \geq 0, \]
or equivalently,
\[ (u_{tt} - 2 b ) \nabla^2 (u - u_0) - d u_t \otimes d u_t \geq 0.  \]
It follows that
\[  u_{tt} \nabla^2 ( u - u_0 ) - d u_t \otimes d u_t \geq 0.  \]
Consequently,
\[ \begin{aligned}
E_u = & u_{tt} \bigg( \nabla^2 u + s du \otimes du + \Big( \gamma \Delta u - \frac{r}{2} |\nabla u|^2 \Big) g + A \bigg) - d u_t \otimes d u_t \\
\geq & u_{tt}  \bigg( \nabla^2 u_0 + s du_0 \otimes du_0 + \Big( \gamma \Delta u_0 - \frac{r}{2} |\nabla u_0|^2 \Big) g + A \bigg) \\
= & u_{tt} W[u_0].
\end{aligned} \]
Since we have assumed that $\lambda\big( W[u_0] \big) \in \Gamma_k$, by the concavity of $\sigma_k^{\frac{1}{k}}$ we have
\begin{equation*}
 \sigma_k^{ij} (E_u) W_{ij} [u_0] \geq k \sigma_k^{1 - \frac{1}{k}} (E_u) \sigma_k^{\frac{1}{k}} \big( W[u_0] \big) > 0.
\end{equation*}

Also, it is straightforward to see that at $(x_0, t_0)$,
\[ W[u] \geq W[u_0].  \]
Therefore, we conclude that at $(x_0, t_0)$,
\[ \begin{aligned}
\mathcal{L} (\Phi) < & (1 + k) \psi(x, t) - 2 b \sigma_k \big( W[u_0] \big).
\end{aligned} \]
By taking $b$ sufficiently large, depending on the upper bound of $\psi$ and the positive lower bound of $\sigma_k \big( W [u_0] \big)$, we can make
\[ \mathcal{L} (\Phi) < 0 \quad \text{  at  } (x_0, t_0). \]
But this is impossible. Hence $\Phi$ can not attain an interior minimum.

Now we choose $c$ sufficiently large depending in addition on the lower bound of $u_1 - u_0$ such that $\Phi(\cdot, 1) \geq 0 = \Phi(\cdot, 0)$. Then we conclude that $\Phi_t (x, 0) \geq 0$ for all $x \in M$, for otherwise $\Phi$ would attain an interior minimum. Hence we obtain $u_t (\cdot, 0) \geq - c$.

Next, we give an upper bound for $u_t(\cdot, 1)$. Consider the test function
\[ \Theta (x, t) = u(x, t) - u_1 (x) - e t^2,  \]
where $e$ is a positive constant to be determined.

Similar as the previous argument, by taking $e$ sufficiently large, depending on the upper bound of $\psi$ and the positive lower bound of $\sigma_k \big( W [u_1] \big)$, we can prove that $\Theta$ can not attain an interior minimum.
Now we choose $e$ further large depending in addition on the upper bound of $u_1 - u_0$ such that $\Theta(x, 1) = - e \leq u_0(x) - u_1 (x) = \Theta(x, 0)$. Then we conclude that $\Theta_t (x, 1) \leq 0$ for all $x \in M$. Hence we obtain $u_t (\cdot, 1) \leq 2 e$.

\end{proof}

\vspace{4mm}

\section{Gradient estimate}~

\vspace{4mm}

\begin{thm} \label{Thm5}
Let $u$ be an admissible solution to \eqref{eq1}--\eqref{eq1-4}. If $r \neq 0$ or $\gamma > 0$, then we have
\[ \max\limits_{M \times [0, 1]} |\nabla u| \leq C, \]
where $C$ is a positive constant depending on $n$, $k$, $\sup\psi$, $\sup \big| \nabla ( \psi^{\frac{1}{k + 1}} ) \big|$, $g$, $|u_0|_{C^1}$, $|u_1|_{C^1}$.
\end{thm}

\begin{proof}
Let $u \in C^3 \big( M \times (0, 1) \big) \cap C^1 \big( M \times [0, 1] \big)$ be an admissible solution of \eqref{eq1}.
We consider the test function
\[  \Phi_1 = |\nabla u|^2 + e^{- \lambda_2 u} + \lambda_3 t (t - 1), \]
where $\lambda_2$ and $\lambda_3$ are constants to be determined.

By direct calculation, we have
\[ \begin{aligned}
\mathcal{M}_{ij} \big( |\nabla u|^2 \big) = 2 u_l \mathcal{M}_{ij} (u_l) + 2 u_{li} u_{lj} + 2 \gamma |\nabla^2 u|^2 \delta_{ij},
\end{aligned} \]
and thus
\begin{equation} \label{eq2-12}
\begin{aligned}
\mathcal{L} \big( |\nabla u|^2 \big) = & 2 u_l \mathcal{L}(u_l) + 2 (1 - k) u_{tt}^{- k} \sigma_k (E_u) |\nabla u_t|^2 \\
& + 2 u_{tt}^{1 - k} \sigma_k^{ij}(E_u) \Big( |\nabla u_t|^2 W_{ij}[u] +  u_{tt} u_{l i} u_{l j} +  \gamma u_{tt} |\nabla^2 u|^2 \delta_{ij} \\
& - u_{ti} u_{lj} u_{lt} - u_{li} u_{lt} u_{tj} - \frac{u_{ti} u_{tj} |\nabla u_t|^2}{u_{tt}} + \frac{u_{ti} u_{tj} |\nabla u_t|^2}{u_{tt}}  \Big) \\
= &  2 u_l \mathcal{L}(u_l) + 2 u_{tt}^{- k} \sigma_k (E_u) |\nabla u_t|^2 \\
& + 2 u_{tt}^{1 - k} \sigma_k^{ij}(E_u) \Big( u_{tt} u_{l i} u_{l j} +  \gamma u_{tt} |\nabla^2 u|^2 \delta_{ij} \\
& - u_{ti} u_{lj} u_{lt} - u_{li} u_{lt} u_{tj} + \frac{u_{ti} u_{tj} |\nabla u_t|^2}{u_{tt}}  \Big),
\end{aligned}
\end{equation}
where $|\nabla^2 u|^2 = \sum\limits_{l m} u_{lm}^2$.

Also, we compute
%\[ (E_u)_{ij} = u_{tt} \Big( u_{ij} + s u_i u_j + \big( \gamma \Delta u - \frac{r}{2} |\nabla u|^2 \big) %\delta_{ij} + A_{ij} \Big) - u_{ti} u_{tj}, \]
\[\begin{aligned}
(E_u)_{ij, l} = & u_{ttl} W_{ij}[u] + u_{tt} \Big( u_{ijl} + s u_{il} u_j + s u_i u_{jl} \\
& + \big( \gamma (\Delta u)_l - r u_m u_{ml} \big) \delta_{ij} + A_{ij, l} \Big)
  - u_{t il} u_{tj} - u_{ti} u_{tjl}.
\end{aligned} \]
Now we differentiate \eqref{eq2} to obtain
\begin{equation} \label{eq2-13}
\begin{aligned}
& (1 - k) u_{tt}^{- k} u_{ttl} \sigma_k (E_u) + u_{tt}^{1 - k} \sigma_k^{ij}(E_u) \bigg( u_{ttl} W_{ij} [u] \\ & + u_{tt} \Big( u_{ijl} + s u_{il} u_j + s u_i u_{jl}
 + \big( \gamma (\Delta u)_l - r u_m u_{ml} \big) \delta_{ij} + A_{ij, l}  \Big) \\
& - u_{til} u_{tj} - u_{ti} u_{tjl} \bigg) = \psi_l.
\end{aligned}
\end{equation}
Comparing with
\[\begin{aligned}
& \mathcal{L}(u_l) = (1 - k) u_{tt}^{- k} \sigma_k (E_u) u_{ltt} \\
& + u_{tt}^{1 - k} \sigma_k^{ij} (E_u) \Big( u_{ltt} W_{ij}[u] + u_{tt} \mathcal{M}_{ij} (u_l) - u_{ti} u_{ltj} - u_{lti} u_{tj} \Big),
\end{aligned} \]
where
\[ \mathcal{M}_{ij} (u_l) = u_{lij} + s u_i u_{lj} + s u_{li} u_j + \Big( \gamma \Delta (u_l) - r \langle \nabla u, \nabla(u_l) \rangle \Big) \delta_{ij}, \]
and in view of
\begin{equation} \label{eq2-28}
\nabla_{ijl} u = \nabla_{lij} u + R_{lij}^m \nabla_m u,
\end{equation}
and
\begin{equation} \label{eq2-29}
\nabla_l \Delta u = \Delta \nabla_l u - \sum\limits_m R_{lmm}^s \nabla_s u,
\end{equation}
equation \eqref{eq2-13} becomes
\begin{equation} \label{eq2-14}
\mathcal{L}(u_l) = \psi_l + u_{tt}^{2 - k} \sigma_k^{ij} (E_u) (R_{lji}^m u_m + \gamma R_{lmm}^s u_s \delta_{ij} - A_{ij, l}).
\end{equation}

Also, by Cauchy-Schwartz inequality, the term in \eqref{eq2-12} can be estimated as
\begin{equation} \label{eq2-23}
\begin{aligned}
& 2 u_{tt}^{1 - k} \sigma_k^{ij}(E_u) \Big( u_{tt} u_{l i} u_{l j} - u_{ti} u_{lj} u_{lt} - u_{li} u_{lt} u_{tj}  + \frac{u_{ti} u_{tj} |\nabla u_t|^2}{u_{tt}}  \Big) \\
\geq &  2 u_{tt}^{1 - k} \sigma_k^{ij}(E_u) \Big( u_{tt} u_{l i} u_{l j} + \frac{u_{ti} u_{tj} |\nabla u_t|^2}{u_{tt}}  \Big) - 2 u_{tt}^{- k} \sigma_k^{ij}(E_u) u_{ti} u_{lt}^2 u_{tj} \\
& - 2 u_{tt}^{2 - k} \sigma_k^{ij}(E_u) u_{li} u_{lj}
= 0 .
\end{aligned}
\end{equation}
Taking \eqref{eq2-23} and \eqref{eq2-14} into \eqref{eq2-12}, we obtain
\begin{equation} \label{eq2-24}
\begin{aligned}
\mathcal{L} \big( |\nabla u|^2 \big)
\geq &  2 u_l \mathcal{L}(u_l) +   2 u_{tt}^{- k} \sigma_k (E_u) |\nabla u_t|^2 \\
= &  2 u_l  \psi_l  +  2 u_l  u_{tt}^{2 - k} \sigma_k^{ij} (E_u) (R_{lji}^m u_m \\
& + \gamma R_{lmm}^s u_s \delta_{ij} - A_{ij, l}) + 2 u_{tt}^{- k} \sigma_k (E_u) |\nabla u_t|^2 .
\end{aligned}
\end{equation}

Next, we compute
\[ \begin{aligned}
\mathcal{M}_{ij} (e^{- \lambda_2 u}) = & e^{- \lambda_2 u} \Big( - \lambda_2 u_{ij} + \lambda_2^2 u_i u_j - 2 \lambda_2 s u_i u_j \\
& + \big( - \gamma \lambda_2 \Delta u + \gamma \lambda_2^2 |\nabla u|^2 + r \lambda_2 |\nabla u|^2 \big) \delta_{ij} \Big).
\end{aligned} \]

\begin{equation} \label{eq2-16}
\begin{aligned}
\mathcal{L}(e^{- \lambda_2 u}) = & - \lambda_2 e^{- \lambda_2 u} \mathcal{L}(u) + (1 - k) u_{tt}^{- k} \sigma_k (E_u) e^{- \lambda_2 u} \lambda_2^2 u_t^2 \\
& + u_{tt}^{1 - k} \sigma_k^{ij} (E_u) e^{- \lambda_2 u} \lambda_2^2 \Big( u_t^2 W_{ij}[u] - u_{tt}^{- 1} u_t^2 u_{ti} u_{tj} + u_i u_j u_{tt} \\
& + \gamma |\nabla u|^2 u_{tt} \delta_{ij} - u_{ti} u_j u_t - u_i u_t u_{tj} + u_{tt}^{- 1} u_t^2 u_{ti} u_{tj} \Big) \\
= &  - \lambda_2 e^{- \lambda_2 u} \mathcal{L}(u) + u_{tt}^{- k} \sigma_k (E_u) e^{- \lambda_2 u} \lambda_2^2 u_t^2 \\
& + u_{tt}^{1 - k} \sigma_k^{ij} (E_u) e^{- \lambda_2 u} \lambda_2^2 \Big(  u_i u_j u_{tt}
 + \gamma |\nabla u|^2 u_{tt} \delta_{ij} \\
& - u_{ti} u_j u_t - u_i u_t u_{tj} + u_{tt}^{- 1} u_t^2 u_{ti} u_{tj} \Big).
\end{aligned}
\end{equation}
Also, we can compute
\[ \mathcal{M}_{ij} (u) = u_{ij} + 2 s u_i u_j + \big( \gamma \Delta u - r |\nabla u|^2 \big) \delta_{ij}. \]
\begin{equation} \label{eq2-18}
\begin{aligned}
\mathcal{L}(u) = & (1 - k) u_{tt}^{1 - k} \sigma_k (E_u) + u_{tt}^{1 - k} \sigma_k^{ij} (E_u) \Big( u_{tt} W_{ij} [u] \\
& + u_{tt} \mathcal{M}_{ij}(u) - 2 u_{ti} u_{tj} \Big)  \\
= &  (1 - k) u_{tt}^{1 - k} \sigma_k (E_u) + u_{tt}^{1 - k} \sigma_k^{ij} (E_u) \Big( 2 u_{tt} W_{ij} [u] - u_{tt} A_{ij} \\
& + s u_{tt} u_i u_j - \frac{r}{2} |\nabla u|^2 u_{tt} \delta_{ij} - 2 u_{ti} u_{tj} \Big) \\
= & (1 + k) \psi - u_{tt}^{2 - k} \sigma_k^{ij} (E_u) A_{ij} + s u_{tt}^{2 - k} \sigma_k^{ij} (E_u) u_i u_j \\
& - \frac{r}{2} |\nabla u|^2 u_{tt}^{2 - k} (n - k + 1) \sigma_{k - 1} (E_u).
\end{aligned}
\end{equation}
By Cauchy-Schwartz inequality, the term in \eqref{eq2-16} can be estimated as
\begin{equation} \label{eq2-21}
\begin{aligned}
& u_{tt}^{1 - k} \sigma_k^{ij} (E_u) \Big(  u_i u_j u_{tt}
 - u_{ti} u_j u_t - u_i u_t u_{tj} + u_{tt}^{- 1} u_t^2 u_{ti} u_{tj} \Big) \\
\geq &  u_{tt}^{1 - k} \sigma_k^{ij} (E_u) \Big(  u_i u_j u_{tt}
 - \frac{1}{2} u_i u_j u_{tt} - 2 u_{tt}^{- 1} u_t^2 u_{ti} u_{tj} + u_{tt}^{- 1} u_t^2 u_{ti} u_{tj} \Big) \\
= &   u_{tt}^{1 - k} \sigma_k^{ij} (E_u) \Big( \frac{1}{2} u_i u_j u_{tt}
 - u_{tt}^{- 1} u_t^2 u_{ti} u_{tj} \Big) \\
= &    u_{tt}^{1 - k} \Big( \frac{1}{2} \sigma_k^{ij} (E_u) u_i u_j u_{tt}
 - u_{tt}^{- 1} \sigma_k^{ij} \big(  u_{tt} W[u] - d u_t \otimes d u_t \big)  u_t^2 u_{ti} u_{tj} \Big) \\
= &  \frac{1}{2} u_{tt}^{1 - k} \sigma_k^{ij} (E_u) u_i u_j u_{tt}
 - u_{tt}^{- 1} \sigma_k^{ij} \big( W[u] \big)  u_t^2 u_{ti} u_{tj} \\
= &   \frac{1}{2} u_{tt}^{2 - k} \sigma_k^{ij} (E_u) u_i u_j
 - \sigma_k \big( W[u] \big)  u_t^2 + u_{tt}^{- k} u_t^2 \sigma_k (E_u).
\end{aligned}
\end{equation}
Taking \eqref{eq2-21} and \eqref{eq2-18} into \eqref{eq2-16} yields
\begin{equation} \label{eq2-22}
\begin{aligned}
& \mathcal{L}(e^{- \lambda_2 u}) \\
\geq &  - \lambda_2 e^{- \lambda_2 u} \mathcal{L}(u) + e^{- \lambda_2 u} \lambda_2^2 \gamma (n - k + 1) u_{tt}^{2 - k} |\nabla u|^2 \sigma_{k - 1} (E_u)  \\
& + e^{- \lambda_2 u} \lambda_2^2 \Big(    \frac{1}{2} u_{tt}^{2 - k} \sigma_k^{ij} (E_u) u_i u_j
 - \sigma_k \big( W[u] \big)  u_t^2 + 2 u_{tt}^{- k} u_t^2 \sigma_k (E_u) \Big) \\
= &   - \lambda_2 e^{- \lambda_2 u} \Big( (1 + k) \psi - u_{tt}^{2 - k} \sigma_k^{ij} (E_u) A_{ij} \\
& + s u_{tt}^{2 - k} \sigma_k^{ij} (E_u) u_i u_j - \frac{r}{2} |\nabla u|^2 u_{tt}^{2 - k} (n - k + 1) \sigma_{k - 1} (E_u) \Big) \\
& + e^{- \lambda_2 u} \lambda_2^2 \gamma (n - k + 1) u_{tt}^{2 - k} |\nabla u|^2 \sigma_{k - 1} (E_u) \\
& + e^{- \lambda_2 u} \lambda_2^2 \Big(    \frac{1}{2} u_{tt}^{2 - k} \sigma_k^{ij} (E_u) u_i u_j
 - \sigma_k \big( W[u] \big)  u_t^2 + 2 u_{tt}^{- k} u_t^2 \sigma_k (E_u) \Big) .
\end{aligned}
\end{equation}

Last, we compute
\begin{equation} \label{eq2-20}
\begin{aligned}
\mathcal{L} \Big(  \lambda_3 t (t - 1) \Big) = & (1 - k) u_{tt}^{- k} \sigma_k (E_u)  2 \lambda_3 + u_{tt}^{1 - k} \sigma_k^{ij}(E_u) 2 \lambda_3 W_{ij}[u] \\
= &  (1 - k) u_{tt}^{- k} \sigma_k (E_u)  2 \lambda_3 + u_{tt}^{- k} \sigma_k^{ij}(E_u) 2 \lambda_3 \Big( (E_u)_{ij} + u_{ti} u_{tj} \Big) \\
= & 2 \lambda_3 u_{tt}^{- k} \sigma_k (E_u) + 2 \lambda_3 u_{tt}^{- k} \sigma_k^{ij} (E_u) u_{ti} u_{tj} \\
= &  2 \lambda_3 u_{tt}^{- k} \sigma_k (E_u) + 2 \lambda_3 u_{tt}^{- k} \sigma_k^{ij} \big( u_{tt} W[u] \big) u_{ti} u_{tj} \\
= & 2 \lambda_3 \sigma_k \big( W[u] \big).
\end{aligned}
\end{equation}

Combining \eqref{eq2-24}, \eqref{eq2-22} and \eqref{eq2-20} yields,
\begin{equation} \label{eq2-25}
\begin{aligned}
& \mathcal{L}(\Phi_1) \geq  2 u_l  \psi_l  +  2 u_l  u_{tt}^{2 - k} \sigma_k^{ij} (E_u) (R_{lji}^m u_m  + \gamma R_{lmm}^s u_s \delta_{ij} - A_{ij, l}) \\
& + 2 u_{tt}^{- k} \sigma_k (E_u) |\nabla u_t|^2 - \lambda_2 e^{- \lambda_2 u} \Big( (1 + k) \psi - u_{tt}^{2 - k} \sigma_k^{ij} (E_u) A_{ij} \\
& + s u_{tt}^{2 - k} \sigma_k^{ij} (E_u) u_i u_j - \frac{r}{2} |\nabla u|^2 u_{tt}^{2 - k} (n - k + 1) \sigma_{k - 1} (E_u) \Big) \\
& + e^{- \lambda_2 u} \lambda_2^2 \gamma (n - k + 1) u_{tt}^{2 - k} |\nabla u|^2 \sigma_{k - 1} (E_u) \\
& + e^{- \lambda_2 u} \lambda_2^2 \Big(    \frac{1}{2} u_{tt}^{2 - k} \sigma_k^{ij} (E_u) u_i u_j
 - \sigma_k \big( W[u] \big)  u_t^2 + 2 u_{tt}^{- k} u_t^2 \sigma_k (E_u) \Big) \\
& + 2 \lambda_3 \sigma_k \big( W[u] \big) \\
\geq &  2 u_l \psi_l - \lambda_2 e^{- \lambda_2 u} (1 + k) \psi + 2 e^{- \lambda_2 u} \lambda_2^2 u_{tt}^{- 1} u_t^2 \psi \\
& + \Big( \lambda_2 e^{- \lambda_2 u}  \frac{(n - k + 1) r}{2} |\nabla u|^2 + e^{- \lambda_2 u} \lambda_2^2 \gamma (n - k + 1) |\nabla u|^2 \\
& - C |\nabla u|^2 - C |\lambda_2| e^{- \lambda_2 u} \Big) u_{tt}^{2 - k} \sigma_{k - 1} (E_u)  \\
& + e^{- \lambda_2 u} \Big( \frac{\lambda_2^2}{2} - \lambda_2 s  \Big) u_{tt}^{2 - k} \sigma_k^{ij} (E_u) u_i u_j
 + \Big( 2 \lambda_3  - C e^{- \lambda_2 u} \lambda_2^2 \Big) \sigma_k \big( W[u] \big).
\end{aligned}
\end{equation}
By arithmetic and geometric mean inequality, Newton-Maclaurin inequality and \eqref{eq2}, we have
\begin{equation*}
\begin{aligned}
& k u_{tt}^{2 - k} \sigma_{k - 1} (E_u) |\nabla u|^2 + \psi  u_{tt}^{- 1} u_t^2
\geq  (k + 1) \Big( u_{tt}^{(2 - k) k - 1} \sigma_{k - 1}^k (E_u) |\nabla u|^{2 k} \psi u_t^2 \Big)^{\frac{1}{k + 1}} \\
\geq & C(n, k) \Big( \psi^{k - 1} |\nabla u|^{2 k} \psi  u_t^2 \Big)^{\frac{1}{k + 1}}
=  C(n, k) \psi^{\frac{k}{k + 1}} |\nabla u|^{\frac{2 k}{k + 1}} |u_t|^{\frac{2}{k + 1}}.
\end{aligned}
\end{equation*}

\vspace{2mm}

{\bf The case when $r \neq 0$.}

\vspace{2mm}

When $r > 0$, we may subtract $c_1 t + c_2$ from $u$, where $c_1$ and $c_2$ are sufficiently large constants, to make $u < 0$ and $u_t \leq - 1$ on $M \times [0, 1]$. When $r < 0$, we may add $c_1 t + c_2$ to $u$, to make $u > 0$ and $u_t \geq 1$ on $M \times [0, 1]$.

When $r > 0$, we choose $\lambda_2 > 0$ sufficiently large, while when $r < 0$, we choose $ - \lambda_2 > 0$ sufficiently large, and then choose $\lambda_3 > 0$ sufficiently large so that \eqref{eq2-25} reduces to
\begin{equation} \label{eq2-27}
\begin{aligned}
& \mathcal{L}(\Phi_1)
\geq - 2 |\nabla u| |\nabla \psi| - \lambda_2 e^{- \lambda_2 u} (1 + k) \psi \\
& + \min\Big\{ 2 |\lambda_2|, \frac{(n - k + 1) |r|}{4 k} \Big\} |\lambda_2| e^{- \lambda_2 u} C(n, k) \psi^{\frac{k}{k + 1}} |\nabla u|^{\frac{2 k}{k + 1}} |u_t|^{\frac{2}{k + 1}} \\
& + \Big( \frac{1}{2} \lambda_2 e^{- \lambda_2 u}  \frac{(n - k + 1) r}{2} |\nabla u|^2 - C |\nabla u|^2 - C |\lambda_2| e^{- \lambda_2 u} \Big) u_{tt}^{2 - k} \sigma_{k - 1} (E_u)  \\
& + e^{- \lambda_2 u} \Big( \frac{\lambda_2^2}{2} - \lambda_2 s  \Big) u_{tt}^{2 - k} \sigma_k^{ij} (E_u) u_i u_j
 + \Big( 2 \lambda_3  - C e^{- \lambda_2 u} \lambda_2^2 \Big) \sigma_k \big( W[u] \big) \\
\geq &  2 |\nabla \psi| \Big( |\nabla u|^{\frac{2 k}{k + 1}} - |\nabla u| \Big) - \lambda_2 e^{- \lambda_2 u} (1 + k) \psi \\
& + \frac{1}{2} \min\Big\{ 2 |\lambda_2|, \frac{(n - k + 1) |r|}{4 k} \Big\} |\lambda_2| e^{- \lambda_2 u} C(n, k) \psi^{\frac{k}{k + 1}} |\nabla u|^{\frac{2 k}{k + 1}}
\\ & + \Big( \lambda_2 e^{- \lambda_2 u}  \frac{(n - k + 1) r}{8} |\nabla u|^2 - C |\lambda_2| e^{- \lambda_2 u} \Big) u_{tt}^{2 - k} \sigma_{k - 1} (E_u) .
\end{aligned}
\end{equation}

Suppose that $\Phi_1$ attains its maximum at $(x_1, t_1) \in M \times (0, 1)$. We may assume that $|\nabla u| (x_1, t_1)$ is sufficiently large, since otherwise we are done, so that \eqref{eq2-27} implies that
\[  \mathcal{L}(\Phi_1) (x_1, t_1)  > 0 .  \]
But this is impossible. Therefore, $\Phi_1$ attains its maximum on $M \times \{ 0, 1 \}$. We hence obtain a bound for $|\nabla u|$ on $M \times [0, 1]$.

\vspace{2mm}

{\bf The case when $\gamma > 0$.}

In this case, we use the term
\[ e^{- \lambda_2 u} \lambda_2^2 \gamma (n - k + 1) u_{tt}^{2 - k} |\nabla u|^2 \sigma_{k - 1} (E_u)\]
instead of
\[  e^{- \lambda_2 u} \lambda_2 \frac{(n - k + 1) r}{2} u_{tt}^{2 - k} |\nabla u|^2 \sigma_{k - 1} (E_u) \]
to derive the estimate. We may subtract $c_1 t + c_2$ from $u$, where $c_1$ and $c_2$ are sufficiently large constants, to make $u < 0$ and $u_t \leq - 1$ on $M \times [0, 1]$.  We choose $\lambda_2 > 0$ sufficiently large, and then choose $\lambda_3 > 0$ sufficiently large so that
\eqref{eq2-25} reduces to
\begin{equation} \label{eq2-26}
\begin{aligned}
\mathcal{L}(\Phi_1)
\geq &  2 |\nabla \psi| \Big( |\nabla u|^{\frac{2 k}{k + 1}} - |\nabla u| \Big) - \lambda_2 e^{- \lambda_2 u} (1 + k) \psi \\
& + \frac{1}{2} \min\Big\{ 2, \frac{(n - k + 1) \gamma}{4 k} \Big\} \lambda_2^2 e^{- \lambda_2 u} C(n, k) \psi^{\frac{k}{k + 1}} |\nabla u|^{\frac{2 k}{k + 1}}
\\ & + \Big( \lambda_2^2 e^{- \lambda_2 u}  \frac{(n - k + 1) \gamma}{4} |\nabla u|^2 - C \lambda_2 e^{- \lambda_2 u} \Big) u_{tt}^{2 - k} \sigma_{k - 1} (E_u).
\end{aligned}
\end{equation}
The rest of the proof follows from the same line as the above case.

\end{proof}

\vspace{4mm}

\section{Second order boundary estimate}

\vspace{4mm}

In this section, we derive boundary estimate for second order derivatives.

\begin{thm} \label{Thm3}
Let $u$ be an admissible solution to \eqref{eq1}--\eqref{eq1-4}. Suppose that $\gamma > 0$ or $r > 0$. We have the estimate
\[ \max\limits_{M \times \{0, 1\}} \Big( u_{tt} + |\nabla u_t| + |\nabla^2 u| \Big) \leq C . \]
\end{thm}

\begin{proof}
A bound for $|\nabla^2 u|$ on $t = 0$ and $t = 1$ is immediate.

Next, we give a bound for $|\nabla u_t|$ on $t = 0$. Consider the test function
\[ \Psi = \big| \nabla (u - u_0) \big| + e^{- a (u - u_0)} - 1 + b t (t - 1) - c t , \]
where $a$, $b$ and $c$ are positive constants to be chosen later.

We shall prove that by choosing $a$, $b$ appropriately, $\Psi$ can not achieve an interior maximum. If not, suppose that $\Psi$ attains an interior maximum at $(x_1, t_1) \in M \times (0, 1)$. We choose a smooth local orthonormal frame field $e_1, \ldots, e_n$ around $x_1$ on $M$ such that
\[ e_1 (x_1) = \frac{\nabla(u - u_0)}{\big| \nabla(u - u_0) \big|} (x_1, t_1) \quad \text{  if  }  \big| \nabla(u - u_0) \big| (x_1, t_1) \neq 0. \]
If $\big| \nabla(u - u_0) \big| (x_1, t_1) = 0$, we may choose arbitrary smooth local orthonormal frame field $e_1, \ldots, e_n$ around $x_1$ on $M$. Then we see that
\[ \tilde{\Psi} = (u - u_0)_1 + e^{- a (u - u_0)} - 1 + b t (t - 1) - c t \]
attains a local maximum at $(x_1, t_1)$.

By \eqref{eq2-14},
\begin{equation} \label{eq4-1}
\begin{aligned}
\mathcal{L} \Big( (u - u_0)_1 \Big) =  \mathcal{L} \big( u_1 \big) - \mathcal{L} \big( (u_0)_1 \big) \geq \psi_1 - C u_{tt}^{2 - k} \sigma_{k - 1} (E_u) .
\end{aligned}
\end{equation}

Next, we compute
\[ \begin{aligned}
\mathcal{M}_{ij} \big( e^{- a (u - u_0) } \big) = &  e^{- a (u - u_0) } \Big( - a \mathcal{M}_{ij} (u - u _0) + a^2 (u - u_0)_i (u - u_0)_j \\
& + \gamma a^2 \big| \nabla ( u - u_0 ) \big|^2 \delta_{ij} \Big).
\end{aligned} \]
It follows that
\begin{equation} \label{eq4-3}
\begin{aligned}
& \mathcal{L} \big( e^{- a (u - u_0) } \big) \\
= & - a e^{- a (u - u_0)} \mathcal{L}(u - u_0) + u_{tt}^{- k} \sigma_k (E_u) e^{- a (u - u_0)} a^2 u_t^2 \\
& + a^2 u_{tt}^{1 - k} \sigma_k^{ij} (E_u) e^{- a (u - u_0)} \Big(  (u - u_0)_i (u - u_0)_j u_{tt} \\
& - 2 u_{ti} (u - u_0)_j u_t + u_{tt}^{- 1} u_t^2 u_{ti} u_{tj} \Big) \\
& + u_{tt}^{2 - k} (n - k + 1) \sigma_{k - 1} (E_u) e^{- a (u - u_0)} \gamma a^2 \big| \nabla (u - u_0) \big|^2.
\end{aligned}
\end{equation}
By Cauchy-Schwartz inequality,
\begin{equation} \label{eq4-7}
\begin{aligned}
& u_{tt}^{1 - k} \sigma_k^{ij} (E_u) \Big(  (u - u_0)_i (u - u_0)_j u_{tt}
- 2 u_{ti} (u - u_0)_j u_t + u_{tt}^{- 1} u_t^2 u_{ti} u_{tj} \Big)  \\
\geq &  u_{tt}^{1 - k} \sigma_k^{ij} (E_u) \Big(  (u - u_0)_i (u - u_0)_j u_{tt}
 - \frac{1}{2} (u - u_0)_i (u - u_0)_j u_{tt} \\
& - 2 u_{tt}^{- 1} u_t^2 u_{ti} u_{tj}+ u_{tt}^{- 1} u_t^2 u_{ti} u_{tj} \Big) \\
= &   u_{tt}^{1 - k} \sigma_k^{ij} (E_u) \Big( \frac{1}{2} (u - u_0)_i (u - u_0)_j u_{tt}
 - u_{tt}^{- 1} u_t^2 u_{ti} u_{tj} \Big) \\
= &  \frac{1}{2} u_{tt}^{2 - k} \sigma_k^{ij} (E_u) (u - u_0)_i (u - u_0)_j
 - u_{tt}^{- 1} \sigma_k^{ij} \big( W[u] \big)  u_t^2 u_{ti} u_{tj} \\
= &  \frac{1}{2} u_{tt}^{2 - k} \sigma_k^{ij} (E_u) (u - u_0)_i (u - u_0)_j
 - \sigma_k \big( W[u] \big)  u_t^2 + u_{tt}^{- k} u_t^2 \sigma_k (E_u).
\end{aligned}
\end{equation}
Also, we notice that
\[ \begin{aligned}
\mathcal{L} (u_0) = u_{tt}^{2 - k} \sigma_k^{ij}(E_u) \Big( (u_0)_{ij} + s u_i (u_0)_j + s (u_0)_i u_j + \big( \gamma \Delta u_0 - r \langle \nabla u, \nabla u_0 \rangle \big) \delta_{ij} \Big).
\end{aligned} \]
Combining with \eqref{eq2-18}, we arrive at
\begin{equation} \label{eq4-8}
\begin{aligned}
& \mathcal{L} (u - u_0) = (1 + k) \psi + s u_{tt}^{2 - k} \sigma_k^{ij} (E_u) (u - u_0)_i (u - u_0)_j \\
& - \frac{r}{2} u_{tt}^{2 - k} (n - k + 1) \sigma_{k - 1} (E_u) \big| \nabla (u - u_0) \big|^2 - u_{tt}^{2 - k} \sigma_k^{ij} (E_u) W_{ij}[u_0] .
\end{aligned}
\end{equation}
By \eqref{eq4-7} and \eqref{eq4-8}, \eqref{eq4-3} can be estimated as
\begin{equation} \label{eq4-4}
\begin{aligned}
& \mathcal{L} \big( e^{- a (u - u_0) } \big) \\
\geq & - a e^{- a (u - u_0)} \Big(  (1 + k) \psi + s u_{tt}^{2 - k} \sigma_k^{ij} (E_u) (u - u_0)_i (u - u_0)_j \\
& - \frac{r}{2} u_{tt}^{2 - k} (n - k + 1) \sigma_{k - 1} (E_u) \big| \nabla (u - u_0) \big|^2 - u_{tt}^{2 - k} \sigma_k^{ij} (E_u) W_{ij}[u_0]  \Big) \\
& + a^2 e^{- a (u - u_0)} \Big( \frac{1}{2} u_{tt}^{2 - k} \sigma_k^{ij} (E_u) (u - u_0)_i (u - u_0)_j - \sigma_k \big( W[u] \big) u_t^2 \\
& + 2 u_{tt}^{- k} u_t^2 \sigma_k (E_u) \Big) + u_{tt}^{2 - k} (n - k + 1) \sigma_{k - 1} (E_u) \gamma a^2  e^{- a (u - u_0) } \big| \nabla (u - u_0) \big|^2.
\end{aligned}
\end{equation}
Also by \eqref{eq2-20},
\begin{equation} \label{eq4-5}
\mathcal{L} \Big( b t (t - 1) \Big)
= 2 b \sigma_k \big( W[u] \big).
\end{equation}
In addition, it is obvious to see that
\[ \mathcal{L} ( - 1 - c t ) = 0. \]

We realize that $\lambda\big( W[u_0] \big) \in \Gamma_k$ and $\Gamma_k$ is open. Thus there exists a small positive constant $c_0$ such that $\lambda\big( W[u_0] - c_0 I \big) \in \Gamma_k$.
It follows that
\[ \begin{aligned}
\sigma_k^{ij} (E_u) W_{ij}[u_0] = & \sigma_k^{ij} (E_u) \big( W_{ij}[u_0] - c_0 \delta_{ij} \big) + c_0 (n - k + 1) \sigma_{k - 1} (E_u) \\
\geq & k \sigma_k^{\frac{1}{k}}\big( W[u_0] - c_0 I \big) \sigma_k^{1 - \frac{1}{k}}(E_u) +  c_0 (n - k + 1) \sigma_{k - 1} (E_u) \\
> & c_0 (n - k + 1) \sigma_{k - 1} (E_u) .
\end{aligned} \]

Combining \eqref{eq4-1}, \eqref{eq4-4}, \eqref{eq4-5} and in view of the above fact, we have
\begin{equation} \label{eq4-9}
\begin{aligned}
& \mathcal{L} ( \tilde{\Psi} ) \geq \psi_1 - C u_{tt}^{2 - k} \sigma_{k - 1} (E_u) \\
&  - a e^{- a (u - u_0)}  (1 + k) \psi - a e^{- a (u - u_0)}  s u_{tt}^{2 - k} \sigma_k^{ij} (E_u) (u - u_0)_i (u - u_0)_j \\
& +  a e^{- a (u - u_0)}  \frac{r}{2} u_{tt}^{2 - k} (n - k + 1) \sigma_{k - 1} (E_u) \big| \nabla (u - u_0) \big|^2 \\
& + a e^{- a (u - u_0)}  u_{tt}^{2 - k}  c_0 (n - k + 1) \sigma_{k - 1} (E_u)  \\
& + a^2 e^{- a (u - u_0)}  \frac{1}{2} u_{tt}^{2 - k} \sigma_k^{ij} (E_u) (u - u_0)_i (u - u_0)_j - a^2 e^{- a (u - u_0)} \sigma_k \big( W[u] \big) u_t^2 \\
& + a^2 e^{- a (u - u_0)}  2 u_{tt}^{- k} u_t^2 \sigma_k (E_u) \\
& + u_{tt}^{2 - k} (n - k + 1) \sigma_{k - 1} (E_u) \gamma a^2  e^{- a (u - u_0) } \big| \nabla (u - u_0) \big|^2 + 2 b \sigma_k \big( W[u] \big).
\end{aligned}
\end{equation}

We may subtract $c_1 t + c_2$ from $u$, where $c_1$ and $c_2$ are sufficiently large constants, to make $u - u_0 < 0$ and $u_t \leq - 1$ on $M \times [0, 1]$. Choosing $a$ sufficiently large, \eqref{eq4-9} reduces to
\begin{equation} \label{eq4-12}
\begin{aligned}
& \mathcal{L} ( \tilde{\Psi} ) \geq \psi_1 - a e^{- a (u - u_0)}  (1 + k) \psi   \\
& +  a e^{- a (u - u_0)}  \frac{r}{2} u_{tt}^{2 - k} (n - k + 1) \sigma_{k - 1} (E_u) \big| \nabla (u - u_0) \big|^2 \\
& + \frac{a}{2} e^{- a (u - u_0)}  u_{tt}^{2 - k}  c_0 (n - k + 1) \sigma_{k - 1} (E_u)  \\
& - a^2 e^{- a (u - u_0)} \sigma_k \big( W[u] \big) u_t^2  + a^2 e^{- a (u - u_0)}  2 u_{tt}^{- k}  \sigma_k (E_u) \\
& + u_{tt}^{2 - k} (n - k + 1) \sigma_{k - 1} (E_u) \gamma a^2  e^{- a (u - u_0) } \big| \nabla (u - u_0) \big|^2 + 2 b \sigma_k \big( W[u] \big).
\end{aligned}
\end{equation}

\vspace{2mm}

{\bf The case when $\gamma > 0$ or  $r > 0$.}

\vspace{2mm}

In these cases, by choosing $a$ further large if necessary, \eqref{eq4-12} further reduces to
\begin{equation} \label{eq4-6}
\begin{aligned}
& \mathcal{L} ( \tilde{\Psi} )
\geq \psi_1
- a e^{- a (u - u_0)} (1 + k) \psi \\
& + \frac{a}{2} e^{- a (u - u_0)} c_0 (n - k + 1) \sigma_{k - 1}(E_u) u_{tt}^{2 - k} - a^2 e^{- a (u - u_0)} \sigma_k \big( W[u] \big) u_t^2 \\
& + 2 a^2 e^{- a (u - u_0)} u_{tt}^{- k} \sigma_k (E_u) + 2 b \sigma_k \big( W[u] \big).
\end{aligned}
\end{equation}
By Newton-Maclaurin inequality, arithmetic and geometric mean inequality, we have
\begin{equation*}
\begin{aligned}
& \frac{a}{2} e^{- a (u - u_0)} c_0 (n - k + 1) \sigma_{k - 1}(E_u) u_{tt}^{2 - k} + 2 a^2 e^{- a (u - u_0)} u_{tt}^{- k} \sigma_k (E_u) \\
\geq & e^{- a (u - u_0)} \Big(  \frac{a}{2} c_0 (n - k + 1) \sigma_{k}^{\frac{k - 1}{k}}(E_u) u_{tt}^{2 - k} + 2 a^2 u_{tt}^{- 1} \psi \Big) \\
= & e^{- a (u - u_0)} \Big(  \frac{a}{2} c_0 (n - k + 1) \psi^{\frac{k - 1}{k}} u_{tt}^{\frac{1}{k}} + 2 a^2 u_{tt}^{- 1} \psi \Big) \\
\geq &  e^{- a (u - u_0)} (k + 1) 2^{- \frac{k - 1}{k + 1}} \Big( \frac{n - k + 1}{k} \Big)^{\frac{k}{k + 1}} a^{\frac{k + 2}{k + 1}} c_0^{\frac{k}{k + 1}} \psi^{\frac{k}{k + 1}} .
\end{aligned}
\end{equation*}
Hence \eqref{eq4-6} reduces to
\begin{equation} \label{eq4-10}
\begin{aligned}
& \mathcal{L} ( \tilde{\Psi} )
\geq  \psi_1
- a e^{- a (u - u_0)} (1 + k) \psi \\
&  + e^{- a (u - u_0)} (k + 1) 2^{- \frac{k - 1}{k + 1}} \Big( \frac{n - k + 1}{k} \Big)^{\frac{k}{k + 1}} a^{\frac{k + 2}{k + 1}} c_0^{\frac{k}{k + 1}} \psi^{\frac{k}{k + 1}} \\
& - a^2 e^{- a (u - u_0)} \sigma_k \big( W[u] \big) u_t^2 + 2 b \sigma_k \big( W[u] \big).
\end{aligned}
\end{equation}
By choosing $a$ further large depending on $\sup\psi$, $\sup \big| \nabla ( \psi^{\frac{1}{k + 1}} ) \big|$, and then choosing $b$ sufficiently large, we have
\[ \mathcal{L} ( \tilde{\Psi} ) > 0  \quad \text{  in a neighborhood of  }  (x_1, t_1) . \]
This means that $\tilde{\Psi}$ can not have an interior maximum at $(x_1, t_1)$. Hence $\Psi$ can not attain its maximum in $M \times (0, 1)$. That is,
\[ \max\limits_{M \times [0, 1]} \Psi =  \max\limits_{M \times \{0, 1\}} \Psi . \]
Now we choose $c$ sufficiently large such that $\Psi (\cdot, 1) \leq 0$.
Hence we have proved that
\[ \Psi \leq \Psi(\cdot, 0) \equiv 0 \quad \text{  on  } M \times [0, 1]. \]

For any point $(x_0, 0) \in M \times \{ t = 0 \}$, we choose a smooth local orthonormal frame field around $x_0$ on $M$. Then in a neighborhood of $(x_0, 0)$, for any $1 \leq l \leq n$,
\[ \begin{aligned}
0 \geq \Psi = &  \big| \nabla (u - u_0) \big| + e^{- a (u - u_0)} - 1 + b t (t - 1) - c t  \\
\geq & \pm (u - u_0)_l +  e^{- a (u - u_0)} - 1 + b t (t - 1) - c t .
\end{aligned}
\]
Since
\[ \Big( \pm (u - u_0)_l +  e^{- a (u - u_0)} - 1 + b t (t - 1) - c t \Big) (x_0, 0) = 0, \]
we thus obtain
\[ \Big( \pm (u - u_0)_l +  e^{- a (u - u_0)} - 1 + b t (t - 1) - c t \Big)_t (x_0, 0) \leq 0, \]
which implies a bound for $|u_{lt}|(x_0, 0)$. Therefore, we have derived a bound for $|\nabla u_t|$ on $t = 0$.

For a bound of $|\nabla u_t|$ on $t = 1$, we consider the test function
\[ \Phi = \big| \nabla (u - u_1) \big| + e^{- a (u - u_1)} - 1 + b t (t - 1) + c (t - 1) , \]
where $a$, $b$ and $c$ are positive constants to be chosen. We can prove similarly as above.

Finally, by \eqref{eq1}, we can directly see that on $t = 0$,
\[ u_{tt} \sigma_k \big( W[u_0] \big) - \sigma_{k}^{ij} \big( W[u_0] \big)  u_{ti} u_{tj} = \psi(x, 0). \]
Since we have obtained a bound for $|\nabla u_t|$ on $t = 0$ and $\sigma_k \big( W[u_0] \big)$ has a positive lower bound, we obtain an upper bound for $u_{tt}$ on $t = 0$. An upper bound for $u_{tt}$ on $t = 1$ can be proved similarly.
\end{proof}

\vspace{4mm}

\section{Global second order estimate}~

\vspace{4mm}

In this section, we write equation \eqref{eq2} in the following form
\begin{equation} \label{eq3}
\begin{aligned}
G(R):= \ln \Big( u_{tt}^{1 - k} \sigma_k (E_u) \Big) = & \ln \Big( u_{tt} \sigma_k (W[u]) -  \sigma_{k}^{ij} (W[u])  u_{ti} u_{tj} \Big)  = \ln \psi,
\end{aligned}
\end{equation}
where
\begin{equation} \label{eq4-11}
R := R_u = (r_{IJ})_{0 \leq I, J \leq n} = \left(
                                         \begin{array}{cccc}
                                           u_{tt} & u_{t1} & \cdots & u_{tn} \\
                                           u_{1t} & W_{11}[u] & \cdots & W_{1n}[u] \\
                                           \vdots & \vdots & \ddots & \vdots \\
                                           u_{nt} & W_{n1}[u] & \cdots & W_{nn}[u] \\
                                         \end{array}
                                       \right).
\end{equation}
The linearized operator of $G(R)$ is given by
\[ \begin{aligned}
\mathbb{L}(v) := & G^{tt} (R) v_{tt} + 2 G^{ti}(R) v_{ti} + G^{ij}(R) \Big( v_{ij} + s u_i v_j + s u_j v_i \\
& + \big( \gamma \Delta v - r \langle \nabla u, \nabla v \rangle \big) \delta_{ij} \Big) \\
= & G^{tt} (R) v_{tt} + 2 G^{ti}(R) v_{ti} + G^{ij}(R) \mathcal{M}_{ij}(v),
\end{aligned} \]
where
\begin{equation} \label{eq3-1}
\begin{aligned}
G^{tt} = & \frac{\partial G}{\partial r_{00}} = \frac{\partial G}{\partial u_{tt}} = \frac{\sigma_k \big( W[u] \big)}{\sigma_k (E_u)} u_{tt}^{k - 1}, \\
G^{ti} = & \frac{\partial G}{\partial r_{0i}} = \frac{\partial G}{\partial u_{ti}} = \frac{- \sigma_k^{ij} \big( W[u] \big) u_{tj}}{u_{tt}^{1 - k} \sigma_k (E_u)} = \frac{- \sigma_k^{ij}(E_u) u_{tj}}{\sigma_k (E_u)}, \quad 1 \leq i \leq n, \\
G^{ij} = & \frac{\partial G}{\partial r_{ij}} = \frac{\partial G}{\partial W_{ij}[u]} = \frac{ u_{tt} \sigma_k^{ij}(E_u)}{\sigma_k (E_u)}, \quad 1 \leq i, j \leq n.
\end{aligned}
\end{equation}

Now we can compute
\begin{equation} \label{eq3-2}
\begin{aligned}
\mathbb{L}(u_{tt}) = & G^{tt} (R) u_{tttt} + 2 G^{ti}(R) u_{ttti} + G^{ij}(R) \Big( u_{ttij} + s u_i u_{ttj} + s u_j u_{tti} \\
& + \big( \gamma \Delta u_{tt} - r \langle \nabla u, \nabla u_{tt} \rangle \big) \delta_{ij} \Big) \\
= & G^{tt} (R) u_{tttt} + 2 G^{ti}(R) u_{ttti} + G^{ij}(R) \mathcal{M}_{ij}(u_{tt}).
\end{aligned}
\end{equation}

Differentiating \eqref{eq3} with respect to $t$ we obtain
\begin{equation*}
G^{IJ} (R) r_{IJt} = G^{tt} (R) u_{ttt} + 2 G^{ti} (R) u_{tit} + G^{ij} (R) \big( W_{ij}[u] \big)_t = \frac{\psi_t}{\psi}.
\end{equation*}
Differentiating again we obtain
\begin{equation} \label{eq3-3}
\begin{aligned}
& G^{IJ, KL} r_{IJ t} r_{KL t} + G^{tt} (R) u_{tttt} + 2 G^{ti} u_{titt} + G^{ij} (R) \big( W_{ij} [u] \big)_{tt} \\
= & \frac{\psi_{tt}}{\psi} - \frac{\psi_t^2}{\psi^2},
\end{aligned}
\end{equation}
where
\[ \begin{aligned}
\big( W_{ij} [u] \big)_t = & u_{ijt} + s u_{it} u_j + s u_i u_{jt} + \Big( \gamma \Delta u_t - r \langle \nabla u, \nabla u_t \rangle \Big) \delta_{ij}, \\
\big( W_{ij} [u] \big)_{tt} = & u_{ijtt} + s u_{itt} u_j + 2 s u_{it} u_{jt} + s u_i u_{jtt} \\
& + \Big( \gamma \Delta u_{tt} - r \langle \nabla u, \nabla u_{tt} \rangle - r |\nabla u_t|^2 \Big) \delta_{ij}.
\end{aligned} \]
By \eqref{eq3-3}, we can see that \eqref{eq3-2} can be expressed as
\begin{equation} \label{eq3-4}
\begin{aligned}
\mathbb{L}(u_{tt}) = & \frac{\psi_{tt}}{\psi} - \frac{\psi_t^2}{\psi^2} - G^{IJ, KL} r_{IJ t} r_{KL t} \\
& - 2 s G^{ij}(R) u _{it} u_{jt} + r |\nabla u_t|^2 \sum G^{ii} (R) \\
 = & \frac{\psi_{tt}}{\psi} - \frac{\psi_t^2}{\psi^2} - G^{IJ, KL} r_{IJ t} r_{KL t} \\
& - 2 s \frac{ u_{tt} \sigma_k^{ij}(E_u)}{\sigma_k (E_u)} u _{it} u_{jt} + r |\nabla u_t|^2 \frac{(n - k + 1) u_{tt} \sigma_{k - 1} (E_u)}{\sigma_k (E_u)} .
\end{aligned}
\end{equation}

In order to give an upper bound for $u_{tt}$, we also need to compute $\mathbb{L}(u_t^2)$, which can be obtained by first computing $\mathcal{L}(u_t^2)$. By direct calculation,
\[ \begin{aligned}
\mathcal{M}_{ij} ( u_t^2 ) = & (u_t^2)_{ij} + s u_i (u_t^2)_j + s (u_t^2)_i u_j + \big( \gamma \Delta (u_t^2) - r \langle \nabla u, \nabla(u_t^2) \rangle \big) \delta_{ij}  \\
= & 2 u_t \mathcal{M}_{ij} (u_t) + 2 u_{ti} u_{tj} + 2 \gamma |\nabla u_t|^2 \delta_{ij},
\end{aligned} \]
and thus
\begin{equation} \label{eq2-11}
\begin{aligned}
\mathcal{L}(u_t^2) = & (1 - k) u_{tt}^{- k} \sigma_k (E_u) (2 u_t u_{ttt} + 2 u_{tt}^2) \\
& + u_{tt}^{1 - k} \sigma_k^{ij}(E_u) \Big( (2 u_t u_{ttt} + 2 u_{tt}^2) W_{ij}[u] + u_{tt} \big( 2 u_t \mathcal{M}_{ij}(u_t) + 2 u_{ti} u_{tj}  \\
& + 2 \gamma |\nabla u_t|^2 \delta_{ij} \big) - u_{ti} (2 u_t u_{ttj} + 2 u_{tj} u_{tt}) - (2 u_t u_{tti} + 2 u_{ti} u_{tt}) u_{tj} \Big) \\
= & 2 u_t \mathcal{L}(u_t) + 2 u_{tt} \psi + 2 \gamma u_{tt}^{1 - k} |\nabla u_t|^2 (n - k + 1) \sigma_{k - 1} (E_u).
\end{aligned}
\end{equation}
In addition, we can compute
\[ \begin{aligned}
& \mathcal{L}(u_t) = (1 - k) u_{tt}^{- k} \sigma_k (E_u) u_{ttt} \\
& + u_{tt}^{1 - k} \sigma_k^{ij} (E_u) \Big( u_{ttt} W_{ij}[u] + u_{tt} \mathcal{M}_{ij}(u_t) - u_{ti} u_{ttj} - u_{tti} u_{tj} \Big),
\end{aligned} \]
where
\[ \mathcal{M}_{ij}(u_t) = u_{tij} + s u_i u_{tj} + s u_{ti} u_j + \big( \gamma \Delta u_t - r \langle \nabla u, \nabla u_t \rangle \big) \delta_{ij}. \]
Differentiating \eqref{eq2} with respect to $t$, we have
\[ (1 - k) u_{tt}^{- k} u_{ttt} \sigma_k (E_u) + u_{tt}^{1 - k} \sigma_k^{ij} (E_u) (E_u)_{ijt} =  \psi_t ,  \]
where
\[ \begin{aligned}
& (E_u)_{ijt} = u_{ttt} W_{ij}[u] + u_{tt} \Big(  u_{ijt} + s u_{it} u_j + s u_i u_{jt} \\
 & + \big( \gamma (\Delta u)_t - r \langle \nabla u, \nabla u_t \rangle \big) \delta_{ij} - u_{tti} u_{tj} - u_{ti} u_{ttj} \Big).
\end{aligned} \]
We notice that
\[ \mathcal{L} (u_t) =  \psi_t . \]
Hence \eqref{eq2-11} becomes
\begin{equation} \label{eq2-17}
\begin{aligned}
\mathcal{L}(u_t^2)
= & 2 u_t  \psi_t  + 2 u_{tt} \psi  + 2 \gamma u_{tt}^{1 - k} |\nabla u_t|^2 (n - k + 1) \sigma_{k - 1} (E_u).
\end{aligned}
\end{equation}

We note that the relation between $\mathcal{L}(v)$ and $\mathbb{L}(v)$ is
\begin{equation} \label{eq3-8}
\mathbb{L}(v) = \frac{\mathcal{L}(v)}{u_{tt}^{1 - k} \sigma_k (E_u)}.
\end{equation}
Therefore, we know that
\begin{equation} \label{eq3-9}
\begin{aligned}
\mathbb{L}(u_t^2)
= & \frac{2 u_t  \psi_t}{\psi} + 2 u_{tt} + \frac{2 \gamma |\nabla u_t|^2 (n - k + 1) \sigma_{k - 1} (E_u)}{\sigma_k (E_u)}.
\end{aligned}
\end{equation}

Also, we notice that for a function $\eta(v)$, we have
\begin{equation} \label{eq3-20}
\begin{aligned}
\mathcal{L} \big( \eta (v) \big) = & \eta' \mathcal{L} (v) + \eta'' v_t^2 \sigma_k \big( W[u] \big) + \eta'' u_{tt}^{2 - k} \sigma_k^{ij} (E_u) v_i v_j \\
& + (n - k + 1) \gamma \eta'' u_{tt}^{2 - k} \sigma_{k - 1} (E_u) |\nabla v|^2 - 2 \eta'' u_{tt}^{1 - k} \sigma_k^{ij} (E_u) u_{ti} v_j v_t.
\end{aligned}
\end{equation}
Therefore,
\begin{equation} \label{eq3-11}
\begin{aligned}
\mathbb{L} \big( \eta(v) \big) =  & \eta' \mathbb{L} (v) + \eta'' \frac{v_t^2 \sigma_k \big( W[u] \big)}{u_{tt}^{1 - k} \sigma_k (E_u)} + \eta'' \frac{u_{tt} \sigma_k^{ij} (E_u) v_i v_j}{\sigma_k(E_u)} \\
& + \frac{(n - k + 1) \gamma \eta'' u_{tt} \sigma_{k - 1} (E_u) |\nabla v|^2}{\sigma_k (E_u)} - \frac{ 2 \eta'' \sigma_k^{ij} (E_u) u_{ti} v_j v_t}{\sigma_k (E_u)}.
\end{aligned}
\end{equation}

\begin{thm} \label{Thm1}
Let $u$ be an admissible solution to \eqref{eq1}--\eqref{eq1-4}. Suppose that (i) $\gamma > 0$ or (ii) $r > 0$ and $2 s k \leq r n$. We have the estimate
\[ \max\limits_{M \times [0, 1]} u_{tt} \leq C. \]
\end{thm}
\begin{proof}
Let $u \in C^4 \big( M \times (0, 1) \big) \cap C^2 \big( M \times [0, 1] \big)$ be an admissible solution of \eqref{eq3}.  We may subtract $c_1 t + c_2$ from $u$, where $c_1$ and $c_2$ are sufficiently large constants, to make $u < 0$ and $u_t \leq - 1$ on $M \times [0, 1]$.

We consider the test function $u_{tt} + \eta(u_t^2)$, where $\eta$ is a function to be chosen later. By \eqref{eq3-4}, \eqref{eq3-9}, \eqref{eq3-11} and the concavity of $G$ we have
\begin{equation} \label{eq5-11}
\begin{aligned}
& \mathbb{L} \Big( u_{tt} + \eta( u_t^2 ) \Big) \\
\geq & \frac{\psi_{tt}}{\psi} - \frac{\psi_t^2}{\psi^2} - 2 s \frac{ u_{tt} \sigma_k^{ij}(E_u)}{\sigma_k (E_u)} u _{it} u_{jt} + r |\nabla u_t|^2 \frac{(n - k + 1) u_{tt} \sigma_{k - 1} (E_u)}{\sigma_k (E_u)} \\
+ & \eta' \bigg(  \frac{2 u_t \psi_t}{\psi} + 2 u_{tt} + \frac{2 \gamma |\nabla u_t|^2 (n - k + 1) \sigma_{k - 1} (E_u)}{\sigma_k (E_u)} \bigg) \\
& + \eta'' \frac{4 u_t^2 u_{tt}^2 \sigma_k \big( W[u] \big)}{u_{tt}^{1 - k} \sigma_k (E_u)} + \eta'' \frac{4 u_t^2 u_{tt} \sigma_k^{ij} (E_u) u_{ti} u_{tj}}{\sigma_k(E_u)} \\
& + \frac{4 (n - k + 1) \gamma \eta'' u_t^2 u_{tt} \sigma_{k - 1} (E_u) |\nabla u_t|^2}{\sigma_k (E_u)} - \frac{ 8 \eta'' u_t^2 \sigma_k^{ij} (E_u) u_{ti} u_{tj} u_{tt}}{\sigma_k (E_u)} \\
=  & \frac{\psi_{tt}}{\psi} - \frac{\psi_t^2}{\psi^2} - 2 s \frac{ u_{tt} \sigma_k^{ij}(E_u)}{\sigma_k (E_u)} u _{it} u_{jt} + r |\nabla u_t|^2 \frac{(n - k + 1) u_{tt} \sigma_{k - 1} (E_u)}{\sigma_k (E_u)} \\
+ & \eta' \bigg( \frac{2 u_t \psi_t}{\psi} + 2 u_{tt} + \frac{2 \gamma |\nabla u_t|^2 (n - k + 1) \sigma_{k - 1} (E_u)}{\sigma_k (E_u)} \bigg) \\
& + 4 \eta'' u_t^2 u_{tt} + \frac{4 (n - k + 1) \gamma \eta'' u_t^2 u_{tt} \sigma_{k - 1} (E_u) |\nabla u_t|^2}{\sigma_k (E_u)}.
\end{aligned}
\end{equation}

{\bf The case when $r > 0$ and $2 s k \leq r n$.} We choose $\eta(v) = v$. Also, we may assume that $E_u = \text{diag} (\lambda_1, \ldots, \lambda_n)$ to prove that

\begin{equation} \label{eq5-17}
\begin{aligned}
& - 2 s \frac{ u_{tt} \sigma_k^{ij}(E_u)}{\sigma_k (E_u)} u _{it} u_{jt} + r |\nabla u_t|^2 \frac{(n - k + 1) u_{tt} \sigma_{k - 1} (E_u)}{\sigma_k (E_u)} \\
= & \bigg( r (n - k + 1) \sum\limits_{i} u_{it}^2 \Big( \sigma_{k - 1} (\lambda | i) + \lambda_i \sigma_{k - 2}(\lambda | i) \Big) - 2 s \sum\limits_i \sigma_{k - 1}(\lambda | i) u _{it}^2 \bigg) \frac{u_{tt}}{\sigma_k (E_u)} \\
= & \sum\limits_{i} u_{it}^2 \bigg( r (n - k + 1) \Big( \sigma_{k - 1} (\lambda | i) + \lambda_i \sigma_{k - 2}(\lambda | i) \Big) - 2 s  \sigma_{k - 1}(\lambda | i) \bigg) \frac{u_{tt}}{\sigma_k (E_u)} \\
\geq & \sum\limits_{i} u_{it}^2 \bigg( \Big( r (n - k + 1) - 2 s \Big) \sigma_{k - 1} (\lambda | i) + r (n - k + 1) \frac{ - \sigma_k (\lambda | i)}{\sigma_{k - 1} (\lambda | i)} \sigma_{k - 2}(\lambda | i)  \bigg) \frac{u_{tt}}{\sigma_k (E_u)} \\
\geq & \sum\limits_{i} u_{it}^2 \bigg( \Big( r (n - k + 1) - 2 s \Big) \sigma_{k - 1} (\lambda | i) - r \frac{(n - k)(k - 1)}{k} \sigma_{k - 1} (\lambda | i) \bigg) \frac{u_{tt}}{\sigma_k (E_u)} \\
= & \sum\limits_{i} u_{it}^2 \bigg( \frac{r n}{k} - 2 s  \bigg) \sigma_{k - 1} (\lambda | i) \frac{u_{tt}}{\sigma_k (E_u)}.
\end{aligned}
\end{equation}
For the last inequality, we have applied Newton-Maclaurin inequality. By requiring that
\begin{equation} \label{eq5-20}
 2 s k \leq r n,
\end{equation}
we know that the above inequality is nonnegative. Hence \eqref{eq5-11} reduces to

\begin{equation} \label{eq5-12}
\begin{aligned}
& \mathbb{L} \Big( u_{tt} + \eta( u_t^2 ) \Big)
\geq  \frac{\psi_{tt}}{\psi} - \frac{\psi_t^2}{\psi^2}
+   \frac{2 u_t \psi_t}{\psi} + 2 u_{tt} .
\end{aligned}
\end{equation}

{\bf The case when $\gamma > 0$.}
We may choose $\eta(v) = \frac{\lambda}{2} v^2$, where $\lambda > 0$ is a constant to be chosen later.
\begin{equation*}
\begin{aligned}
& \mathbb{L} \Big( u_{tt} + \eta( u_t^2 ) \Big) \\
\geq & \frac{\psi_{tt}}{\psi} - \frac{\psi_t^2}{\psi^2} - \big( 2 |s| + |r| \big) \frac{ u_{tt} (n - k + 1) \sigma_{k - 1} (E_u)}{\sigma_k (E_u)} |\nabla u_t|^2 \\
+ & \lambda u_t^2 \bigg( \frac{2 u_t \psi_t}{\psi} + 2 u_{tt} + \frac{2 \gamma |\nabla u_t|^2 (n - k + 1) \sigma_{k - 1} (E_u)}{\sigma_k (E_u)} \bigg) \\
& + 4 \lambda u_t^2 u_{tt} + \frac{4 (n - k + 1) \gamma \lambda u_t^2 u_{tt} \sigma_{k - 1} (E_u) |\nabla u_t|^2}{\sigma_k (E_u)}.
\end{aligned}
\end{equation*}
Now we may choose $\lambda > 0$ sufficiently large so that
\begin{equation} \label{eq5-13}
\begin{aligned}
\mathbb{L} \Big( u_{tt} + \eta( u_t^2 ) \Big)
\geq  \frac{\psi_{tt}}{\psi} - \frac{\psi_t^2}{\psi^2}
+ \lambda u_t^2 \Big( \frac{2 u_t \psi_t}{\psi} + 2 u_{tt}  \Big).
\end{aligned}
\end{equation}

Suppose that $u_{tt} + \eta (u_t^2)$ attains its maximum at $(x_2, t_2) \in M \times (0, 1)$. We may assume that $u_{tt} (x_2, t_2) > 0$ is sufficiently large (otherwise we are done) such that from both \eqref{eq5-12} and \eqref{eq5-13} we can deduce that
\begin{equation*}
\mathbb{L} \Big( u_{tt} + \eta ( u_t^2 ) \Big) (x_2, t_2) > 0 .
\end{equation*}
But this is impossible. We thus obtain the upper bound for $u_{tt}$ on $M \times [0, 1]$.
\end{proof}

Next, we compute $\mathbb{L}( \Delta u )$.
\begin{equation} \label{eq3-5}
\begin{aligned}
\mathbb{L}(\Delta u) = & G^{tt} (R) (\Delta u)_{tt} + 2 G^{ti}(R) (\Delta u)_{ti} + G^{ij}(R) \Big( (\Delta u)_{ij} \\
& + s u_i (\Delta u)_j + s u_j (\Delta u)_i + \big( \gamma \Delta (\Delta u) - r \langle \nabla u, \nabla (\Delta u) \rangle \big) \delta_{ij} \Big).
\end{aligned}
\end{equation}

Taking covariant derivative of \eqref{eq3} in the $e_p$ direction we obtain
\[ G^{IJ}(R) r_{IJ, p} = G^{tt}(R) u_{ttp} + 2 G^{ti} (R) u_{tip} + G^{ij} (R) \big( W_{ij} [u] \big)_p = \frac{\psi_p}{\psi}. \]
Differentiating again we have
\begin{equation} \label{eq3-6}
\begin{aligned}
& G^{IJ, KL} (R) r_{IJ, p} r_{KL, p} + G^{tt} (R) \Delta (u_{tt}) + 2 G^{ti} (R) \Delta (u_{ti}) \\
& + G^{ij} (R) \Delta \big( W_{ij}[u] \big)
=  \frac{\Delta \psi}{\psi} - \frac{|\nabla \psi|^2}{\psi^2},
\end{aligned}
\end{equation}
where
\[  \begin{aligned}
\big( W_{ij}[u] \big)_p = & u_{ijp} + s u_{ip} u_j + s u_i u_{jp} + \Big( \gamma (\Delta u)_p - r u_k u_{k p} \Big) \delta_{ij} + A_{ij, p}, \\
\Delta \big( W_{ij}[u] \big) = & \Delta (u_{ij}) + s \Delta (u_i) u_j + 2 s u_{ip} u_{jp} + s u_i \Delta (u_j) \\
& + \Big( \gamma \Delta (\Delta u) - r u_k \Delta (u_k) - r |\nabla^2 u|^2 \Big) \delta_{ij} + \Delta (A_{ij}).
\end{aligned} \]
In view of \eqref{eq3-6}, \eqref{eq2-28}, \eqref{eq2-29} as well as
\[ \begin{aligned}
\nabla_{ijkl} u = & \nabla_{klij} u + R_{kjl}^m \nabla_{im} u + \nabla_i R_{kjl}^m \nabla_m u + R_{kil}^m \nabla_{mj} u \\
& + R_{kij}^m \nabla_{lm} u + R_{lij}^m \nabla_{km} u + \nabla_k R_{lij}^m \nabla_m u,
\end{aligned} \]
which implies that
\[ \begin{aligned}
\nabla_{ji} ( \Delta u ) = & \Delta ( \nabla_{ji} u ) + R_{lil}^m \nabla_{jm} u + \nabla_j R_{lil}^m \nabla_m u + R_{ljl}^m \nabla_{mi} u \\
& + 2 R_{lji}^m \nabla_{lm} u + \nabla_l R_{lji}^m \nabla_m u,
\end{aligned} \]
\eqref{eq3-5} can be expressed as
\begin{equation} \label{eq3-7}
\begin{aligned}
\mathbb{L}(\Delta u) = & \frac{\Delta \psi}{\psi} - \frac{|\nabla \psi|^2}{\psi^2}  -  G^{IJ, KL} (R) r_{IJ, m} r_{KL, m} \\
& - 2 G^{ti}(R) R_{ill}^m u_{t m} + G^{ij}(R) \Big( R_{ljl}^m u_{mi} + R_{lil}^m u_{mj} + 2 R_{lji}^m u_{ml} \\
& + R_{lil, j}^m u_m + R_{lji, l}^m u_m - s R_{ill}^m u_m u_j - s u_i R_{jll}^m u_m - 2 s u_{im} u_{jm} \\
& + \big( r u_k R_{kll}^m u_m + r |\nabla^2 u|^2 \big) \delta_{ij} - \Delta (A_{ij}) \Big).
\end{aligned}
\end{equation}

\begin{thm} \label{Thm2}
Let $u$ be an admissible solution to \eqref{eq1}--\eqref{eq1-4}. Suppose that (i) $\gamma > 0$ or (ii) $r > 0$ and $2 s k \leq r n$. We have the estimate
\[ \max\limits_{M \times [0, 1]} \Delta u \leq C. \]
\end{thm}
\begin{proof}
Let $u \in C^4 \big( M \times (0, 1) \big) \cap C^2 \big( M \times [0, 1] \big)$ be an admissible solution of \eqref{eq3}.
In view of \eqref{eq3-1} and the concavity of $G$, \eqref{eq3-7} can be estimated as
\begin{equation} \label{eq3-14}
\begin{aligned}
\mathbb{L} (\Delta u) \geq & \frac{\Delta \psi}{\psi} - \frac{|\nabla \psi|^2}{\psi^2} - C \frac{\sigma_{k - 1} (E_u)}{\sigma_k (E_u)} |\nabla u_t|^2 - C \frac{u_{tt} \sigma_{k - 1} (E_u)}{\sigma_k (E_u)} \Big( |\nabla^2 u| + 1 \Big) \\
& + \Big( r |\nabla^2 u|^2 (n - k + 1) \sigma_{k - 1}(E_u) - 2 s \sigma_k^{ij} (E_u) u_{im} u_{jm} \Big) \frac{u_{tt}}{\sigma_k (E_u)}.
\end{aligned}
\end{equation}
Since $\lambda (E_u) \in \Gamma_k \subset \Gamma_1$, we know that
\begin{equation*}
u_{tt} \text{tr} W[u] - |\nabla u_t|^2 > 0.
\end{equation*}
Hence
\begin{equation} \label{eq3-13}
|\nabla u_t|^2 \leq u_{tt} \Big( (1 + \gamma n) \Delta u  + C \Big)  \leq C u_{tt} \Big( |\nabla^2 u|  + 1 \Big) .
\end{equation}
Consequently, \eqref{eq3-14} reduces to
\begin{equation} \label{eq3-15}
\begin{aligned}
\mathbb{L} (\Delta u) \geq & \frac{\Delta \psi}{\psi} - \frac{|\nabla \psi|^2}{\psi^2} - C \frac{u_{tt} \sigma_{k - 1} (E_u)}{\sigma_k (E_u)} \Big( |\nabla^2 u| + 1 \Big) \\
& + \Big( r |\nabla^2 u|^2 (n - k + 1) \sigma_{k - 1}(E_u) - 2 s \sigma_k^{ij} (E_u) u_{im} u_{jm} \Big) \frac{u_{tt}}{\sigma_k (E_u)}.
\end{aligned}
\end{equation}

Now we consider the test function $\Delta u + \lambda |\nabla u|^2 + \mu t (t - 1)$, where $\lambda$, $\mu$ are nonnegative constants to be chosen later.

By \eqref{eq2-20} and \eqref{eq3-8},
\begin{equation} \label{eq3-17}
\begin{aligned}
\mathbb{L} \Big( t (t - 1) \Big)
= \frac{ 2 \sigma_k \big( W[u] \big)}{u_{tt}^{1 - k} \sigma_k (E_u)}.
\end{aligned}
\end{equation}

\vspace{2mm}

{\bf The case when $r > 0$, $s > 0$, $2 s k \leq r n$ and $0 \leq \gamma \leq \frac{1}{2 n}$.}

\vspace{2mm}

Taking \eqref{eq2-14} into \eqref{eq2-12} and by \eqref{eq3-8}, we have

\begin{equation} \label{eq5-15}
\begin{aligned}
& \mathbb{L} \big( |\nabla u|^2 \big) \\
\geq & \frac{ 2 u_l \psi_l}{\psi} + \frac{2 u_l u_{tt} \sigma_k^{ij} (E_u) (R_{lji}^m u_m + \gamma R_{lmm}^s u_s \delta_{ij} - A_{ij, l}) }{\sigma_k (E_u)} + \frac{2 |\nabla u_t|^2 }{u_{tt}} \\
& + \frac{ 2 \sigma_k^{ij}(E_u) }{\sigma_k (E_u)} \bigg( u_{tt} u_{l i} u_{l j}  - u_{ti} u_{lj} u_{lt} - u_{li} u_{lt} u_{tj} + \frac{u_{ti} u_{tj} |\nabla u_t|^2}{u_{tt}} \bigg) \\
\geq & \frac{ 2 u_l \psi_l}{\psi}  - C  u_{tt} \frac{\sigma_{k - 1} (E_u)}{\sigma_k (E_u)} + \frac{2 |\nabla u_t|^2 }{u_{tt}} \\
& + \frac{ 2 \sigma_k^{ij}(E_u) }{\sigma_k (E_u)} \bigg( u_{tt} u_{l i} u_{l j}  - u_{ti} u_{lj} u_{lt} - u_{li} u_{lt} u_{tj} + \frac{u_{ti} u_{tj} |\nabla u_t|^2}{u_{tt}} \bigg) .
\end{aligned}
\end{equation}
By \eqref{eq3-15}, \eqref{eq5-15} and \eqref{eq3-17}, we have
\begin{equation} \label{eq5-16}
\begin{aligned}
& \mathbb{L} \Big( \Delta u + \lambda |\nabla u|^2 + \mu t (t - 1) \Big) \\
\geq & \frac{\Delta \psi}{\psi} - \frac{|\nabla \psi|^2}{\psi^2} - C \frac{u_{tt} \sigma_{k - 1} (E_u)}{\sigma_k (E_u)} \Big( |\nabla^2 u| + 1 \Big) \\
& + \Big( r |\nabla^2 u|^2 (n - k + 1) \sigma_{k - 1}(E_u) - 2 s \sigma_k^{ij} (E_u) u_{im} u_{jm} \Big) \frac{u_{tt}}{\sigma_k (E_u)} \\
& + \frac{ 2 \mu \sigma_k \big( W[u] \big)}{u_{tt}^{1 - k} \sigma_k (E_u)} + \frac{ 2 \lambda u_l \psi_l}{\psi}  - C \lambda  u_{tt} \frac{\sigma_{k - 1} (E_u)}{\sigma_k (E_u)} + \frac{2 \lambda |\nabla u_t|^2 }{u_{tt}} \\
& + \frac{ 2 \lambda \sigma_k^{ij}(E_u) }{\sigma_k (E_u)} \bigg( u_{tt} u_{l i} u_{l j}  - u_{ti} u_{lj} u_{lt} - u_{li} u_{lt} u_{tj} + \frac{u_{ti} u_{tj} |\nabla u_t|^2}{u_{tt}} \bigg).
\end{aligned}
\end{equation}

We choose $\mu = \mu_1 + \mu_2$, where
\[ \mu_1 = \frac{1}{2} \max\limits_{M \times [0, 1]} \frac{|\nabla u_t|^2}{u_{tt}} \]
as defined in \cite{HeXuZhang}, and $\mu_2$ is a positive constant to be chosen later.

We notice that
\begin{equation*}
\begin{aligned}
\frac{\sigma_k \big( W[u] \big)}{u_{tt}^{1 - k} \sigma_k (E_u)}
\geq \frac{\sigma_k^{ij} \big( W[u] \big) u_{ti} u_{tj}}{u_{tt}^{2 - k} \sigma_k (E_u)}
= \frac{\sigma_k^{ij} \big( E_u \big) u_{ti} u_{tj}}{u_{tt} \sigma_k (E_u)} .
\end{aligned}
\end{equation*}
By Theorem \ref{Thm1} we know that $u_{tt} \leq C_0$. We may require that $C_0 \geq \frac{1}{2 s}$ if $r > 0$ and $s > 0$. Then we have

\begin{equation} \label{eq5-18}
\begin{aligned}
& \frac{2 \mu_1 \sigma_k \big( W[u] \big)}{u_{tt}^{1 - k} \sigma_k (E_u)} + \frac{ 2 \lambda \sigma_k^{ij}(E_u) }{\sigma_k (E_u)} \bigg( u_{tt} u_{l i} u_{l j} - u_{ti} u_{lj} u_{lt} - u_{li} u_{lt} u_{tj} + \frac{u_{ti} u_{tj} |\nabla u_t|^2}{u_{tt}} \bigg) \\
\geq & \frac{ 2 \lambda \sigma_k^{ij}(E_u) }{\sigma_k (E_u)} \bigg( u_{tt} u_{l i} u_{l j} - u_{ti} u_{lj} u_{lt} - u_{li} u_{lt} u_{tj} + \Big( 1 + \frac{1}{2 \lambda C_0} \Big) \frac{u_{ti} u_{tj} |\nabla u_t|^2}{u_{tt}} \bigg) \\
\geq &  \frac{ 2 \lambda }{2 \lambda C_0 + 1} \frac{u_{tt} \sigma_k^{ij}(E_u) u_{l i} u_{l j}}{\sigma_k (E_u)}.
\end{aligned}
\end{equation}

Similar to \eqref{eq5-17}, we may assume that $E_u = \text{diag} (\lambda_1, \ldots, \lambda_n)$ and obtain
\begin{equation} \label{eq5-19}
\begin{aligned}
&  \frac{k}{n} |\nabla^2 u|^2 (n - k + 1) \sigma_{k - 1}(E_u) - \sigma_k^{ij} (E_u) u_{im} u_{jm}  \\
= &  \frac{k}{n} (n - k + 1) \sum\limits_{im} u_{im}^2 \Big( \sigma_{k - 1}(\lambda | i) + \lambda_i \sigma_{k - 2} (\lambda | i) \Big) - \sum\limits_{im} \sigma_{k - 1}(\lambda | i) u_{im}^2 \\
= & \sum\limits_{im} u_{im}^2  \bigg( \Big( \frac{k}{n} (n - k + 1) - 1 \Big) \sigma_{k - 1}(\lambda | i) +  \frac{k}{n} (n - k + 1) \lambda_i \sigma_{k - 2} (\lambda | i)  \bigg) \\
\geq & \sum\limits_{im} \bigg( \Big(   \frac{k}{n} (n - k + 1) - 1 \Big) \sigma_{k - 1}(\lambda | i) - \frac{(n - k)(k - 1)}{n} \sigma_{k - 1} (\lambda | i)  \bigg) u_{im}^2 = 0.
\end{aligned}
\end{equation}
By \eqref{eq5-18}, \eqref{eq5-19} and \eqref{eq5-20}, we know that
\begin{equation*}
\begin{aligned}
& \Big( r |\nabla^2 u|^2 (n - k + 1) \sigma_{k - 1}(E_u) - 2 s \sigma_k^{ij} (E_u) u_{im} u_{jm} \Big) \frac{u_{tt}}{\sigma_k (E_u)} \\
& + \frac{ 2 \mu_1 \sigma_k \big( W[u] \big)}{u_{tt}^{1 - k} \sigma_k (E_u)} + \frac{ 2 \lambda \sigma_k^{ij}(E_u) }{\sigma_k (E_u)} \bigg( u_{tt} u_{l i} u_{l j}  - u_{ti} u_{lj} u_{lt} - u_{li} u_{lt} u_{tj} + \frac{u_{ti} u_{tj} |\nabla u_t|^2}{u_{tt}} \bigg) \\
\geq &  \bigg( r |\nabla^2 u|^2 (n - k + 1) \sigma_{k - 1}(E_u) - \Big( 2 s - \frac{ 2 \lambda }{2 \lambda C_0 + 1} \Big) \sigma_k^{ij}(E_u) u_{l i} u_{l j} \bigg) \frac{u_{tt}}{\sigma_k (E_u)} \\
\geq &  \bigg( r - \frac{k}{n} \Big( 2 s - \frac{ 2 \lambda }{2 \lambda C_0 + 1} \Big) \bigg) |\nabla^2 u|^2 (n - k + 1) \sigma_{k - 1}(E_u) \frac{u_{tt}}{\sigma_k (E_u)} \\
\geq &   \frac{k (n - k + 1)}{n} \frac{ 2 \lambda }{2 \lambda C_0 + 1}  |\nabla^2 u|^2 \sigma_{k - 1}(E_u) \frac{u_{tt}}{\sigma_k (E_u)} .
\end{aligned}
\end{equation*}
Hence \eqref{eq5-16} reduces to
\begin{equation} \label{eq5-21}
\begin{aligned}
& \mathbb{L} \Big( \Delta u + \lambda |\nabla u|^2 + \mu t (t - 1) \Big) \\
\geq & \frac{\Delta \psi}{\psi} - \frac{|\nabla \psi|^2}{\psi^2} - C \frac{u_{tt} \sigma_{k - 1} (E_u)}{\sigma_k (E_u)} \Big( |\nabla^2 u| + 1 \Big)  \\
& + \frac{ 2 \mu_2 \sigma_k \big( W[u] \big)}{u_{tt}^{1 - k} \sigma_k (E_u)} + \frac{ 2 \lambda u_l \psi_l}{\psi}  - C \lambda  u_{tt} \frac{\sigma_{k - 1} (E_u)}{\sigma_k (E_u)} \\
& + \frac{k (n - k + 1)}{n} \frac{ 2 \lambda }{2 \lambda C_0 + 1}  |\nabla^2 u|^2 \sigma_{k - 1}(E_u) \frac{u_{tt}}{\sigma_k (E_u)} .
\end{aligned}
\end{equation}
By \eqref{eq4} we have
\begin{equation}  \label{eq5-22}
\frac{\sigma_k \big( W[u] \big)}{u_{tt}^{1 - k} \sigma_k (E_u)} = u_{tt}^{- 1} \frac{ u_{tt} \sigma_k \big( W[u] \big)}{u_{tt}^{1 - k} \sigma_k (E_u)} \geq C_0^{- 1}.
\end{equation}
Choosing $\lambda = 1$ and $\mu_2$ sufficiently large such that \eqref{eq5-21} reduces to
\begin{equation} \label{eq5-23}
\begin{aligned}
& \mathbb{L} \Big( \Delta u + \lambda |\nabla u|^2 + \mu t (t - 1) \Big) \\
\geq &  \frac{u_{tt} \sigma_{k - 1} (E_u)}{\sigma_k (E_u)} \bigg( \frac{k (n - k + 1)}{n} \frac{ 2 }{2 C_0 + 1} |\nabla^2 u|^2 - C \Big( |\nabla^2 u| + 1 \Big) \bigg).
\end{aligned}
\end{equation}

Suppose that $\Delta u + \lambda |\nabla u|^2 + \mu t (t - 1)$ attains its maximum at $(x_3, t_3) \in M \times (0, 1)$. If
\begin{equation} \label{eq5-24}
|\nabla^2 u| (x_3, t_3) > \frac{C + \sqrt{C^2 + 4 \frac{k (n - k + 1)}{n} \frac{ 2 }{2 C_0 + 1} C}}{2 \frac{k (n - k + 1)}{n} \frac{ 2 }{2 C_0 + 1} } := C_1 ,
\end{equation}
then
\begin{equation*}
\bigg( \frac{k (n - k + 1)}{n} \frac{ 2 }{2 C_0 + 1} |\nabla^2 u|^2 - C \Big( |\nabla^2 u| + 1 \Big) \bigg) (x_3, t_3) > 0,
\end{equation*}
or equivalently
\begin{equation*}
\mathbb{L} \Big( \Delta u + \lambda |\nabla u|^2 + \mu t (t - 1) \Big) (x_3, t_3) > 0 .
\end{equation*}
This is impossible. If
\begin{equation*}
|\nabla^2 u| (x_3, t_3) \leq C_1,
\end{equation*}
then on $M \times [0, 1]$,
\begin{equation*}
 \Delta u - \mu \leq \Big( \Delta u + \lambda |\nabla u|^2 + \mu t (t - 1) \Big) (x_3, t_3) \leq \sqrt{n} C_1  + \lambda C .
\end{equation*}
By \eqref{eq3-13} we can further obtain
\begin{equation*}
 \sqrt{n} C_1  + \lambda C + \mu \geq \Delta u  \geq \frac{|\nabla u_t|^2}{(1 + \gamma n) u_{tt}} - C .
\end{equation*}
In particular we have
\begin{equation*}
 \sqrt{n} C_1  + \lambda C + \mu_1 + \mu_2  \geq \frac{2 \mu_1}{1 + \gamma n} - C \geq \frac{4 \mu_1}{3} - C .
\end{equation*}
We thus obtain a uniform upper bound for $\mu_1$, which implies an upper bound for $\Delta u$ on $M \times [0, 1]$. Since $\lambda \big( W[u] \big) \in \Gamma_k \subset \Gamma_1$, we know that $\Delta u \geq - C$.

\vspace{2mm}

{\bf The case when $r > 0$ and $s \leq 0$.}

\vspace{2mm}

We choose $\lambda = 0$ and $\mu$ sufficiently large such that \eqref{eq5-16} reduces to
\begin{equation} \label{eq5-25}
\begin{aligned}
& \mathbb{L} \Big( \Delta u + \lambda |\nabla u|^2 + \mu t (t - 1) \Big) \\
\geq & - C \frac{u_{tt} \sigma_{k - 1} (E_u)}{\sigma_k (E_u)} \Big( |\nabla^2 u| + 1 \Big)  +  r |\nabla^2 u|^2 (n - k + 1) \sigma_{k - 1}(E_u) \frac{u_{tt}}{\sigma_k (E_u)} .
\end{aligned}
\end{equation}

Suppose that $\Delta u + \lambda |\nabla u|^2 + \mu t (t - 1)$ attains its maximum at $(x_3, t_3) \in M \times (0, 1)$. We may assume that $|\nabla^2 u| (x_3, t_3)$ is sufficiently large (otherwise we are done) such that from \eqref{eq5-25} we can deduce that
\begin{equation*}
\mathbb{L} \Big( \Delta u + \lambda |\nabla u|^2 + \mu t (t - 1) \Big) (x_3, t_3) > 0 .
\end{equation*}
But this is impossible. We thus obtain an upper bound for $\Delta u$ on $M \times [0, 1]$.

\vspace{2mm}

{\bf The case when $\gamma > 0$.}

\vspace{2mm}

Taking \eqref{eq2-23}, \eqref{eq2-14} into \eqref{eq2-12} and by \eqref{eq3-8}, we have
\begin{equation} \label{eq3-12}
\begin{aligned}
& \mathbb{L} \big( |\nabla u|^2 \big)
\geq \frac{2 u_l \psi_l}{\psi} + \frac{2 (n - k + 1) \gamma u_{tt} \sigma_{k - 1} (E_u) |\nabla^2 u|^2}{\sigma_k (E_u)} \\
& + 2 u_l  u_{tt} \frac{\sigma_k^{ij} (E_u) (R_{lji}^m u_m + \gamma R_{lmm}^s u_s \delta_{ij} - A_{ij, l}) }{\sigma_k (E_u)}+ 2 u_{tt}^{- 1} |\nabla u_t|^2 \\
\geq & \frac{2 u_l \psi_l}{\psi} + \frac{2 (n - k + 1) \gamma u_{tt} \sigma_{k - 1} (E_u) |\nabla^2 u|^2}{\sigma_k (E_u)} - C  u_{tt} \frac{\sigma_{k - 1} (E_u)}{\sigma_k (E_u)} .
\end{aligned}
\end{equation}

By \eqref{eq3-15}, \eqref{eq3-12} and \eqref{eq3-17},

\begin{equation} \label{eq5-14}
\begin{aligned}
& \mathbb{L} \Big( \Delta u + \lambda |\nabla u|^2 + \mu t (t - 1) \Big) \\
\geq & \frac{\Delta \psi}{\psi} - \frac{|\nabla \psi|^2}{\psi^2} - C \frac{u_{tt} \sigma_{k - 1} (E_u)}{\sigma_k (E_u)} \Big( |\nabla^2 u| + 1 \Big) \\
& + \Big( r |\nabla^2 u|^2 (n - k + 1) \sigma_{k - 1}(E_u) - 2 s \sigma_k^{ij} (E_u) u_{im} u_{jm} \Big) \frac{u_{tt}}{\sigma_k (E_u)} \\
& + \frac{2 \lambda u_l \psi_l}{\psi} + \frac{2 (n - k + 1) \lambda \gamma u_{tt} \sigma_{k - 1} (E_u) |\nabla^2 u|^2}{\sigma_k (E_u)} - \frac{C  \lambda  u_{tt} \sigma_{k - 1} (E_u)}{\sigma_k (E_u)} \\
& + \frac{ 2 \mu \sigma_k \big( W[u] \big)}{u_{tt}^{1 - k} \sigma_k (E_u)}.
\end{aligned}
\end{equation}

We may choose $\lambda$ sufficiently large so that
\begin{equation} \label{eq3-18}
\begin{aligned}
& \mathbb{L} \Big( \Delta u + \lambda |\nabla u|^2 + \mu t (t - 1) \Big) \\
\geq & \frac{\Delta \psi}{\psi} - \frac{|\nabla \psi|^2}{\psi^2} + \frac{2 \lambda u_l \psi_l}{\psi} + \Big( (n - k + 1) \gamma \lambda |\nabla^2 u|^2 - C |\nabla^2 u| \\
& - C ( 1 + \lambda ) \Big) \frac{u_{tt} \sigma_{k - 1} (E_u)}{\sigma_k (E_u)} + \frac{ 2 \mu \sigma_k \big( W[u] \big)}{u_{tt}^{1 - k} \sigma_k (E_u)}.
\end{aligned}
\end{equation}
In view of \eqref{eq5-22}, we can choose $\mu$ sufficiently large so that \eqref{eq3-18} reduces to
\begin{equation} \label{eq3-19}
\begin{aligned}
& \mathbb{L} \Big( \Delta u + \lambda |\nabla u|^2 + \mu t (t - 1) \Big) \\
\geq & \Big( (n - k + 1) \gamma \lambda |\nabla^2 u|^2 - C |\nabla^2 u| - C ( 1 + \lambda ) \Big) \frac{u_{tt} \sigma_{k - 1} (E_u)}{\sigma_k (E_u)}.
\end{aligned}
\end{equation}

Suppose that $\Delta u + \lambda |\nabla u|^2 + \mu t (t - 1)$ attains its maximum at $(x_3, t_3) \in M \times (0, 1)$. We may assume that $|\nabla^2 u| (x_3, t_3)$ is sufficiently large (otherwise we are done) such that from \eqref{eq3-19} we can deduce that
\begin{equation*}
\mathbb{L} \Big( \Delta u + \lambda |\nabla u|^2 + \mu t (t - 1) \Big) (x_3, t_3) > 0 .
\end{equation*}
But this is impossible. We thus obtain an upper bound for $\Delta u$ on $M \times [0, 1]$.

\end{proof}

For $k \geq 2$, by the relation that
\[ \big| W[u] \big|^2 = \sigma_1^2 \big( W[u] \big) - 2 \sigma_2 \big( W[u] \big) \leq  \sigma_1^2 \big( W[u] \big) , \]
we obtain a bound for $|\nabla^2 u|$ on $M \times [0, 1]$.

Finally, by the fact that $\lambda(E_u) \in \Gamma_k \subset \Gamma_1$, we have
\[ u_{tt} \sigma_1 \big( W[u] \big) - |\nabla u_t|^2 > 0, \]
and we therefore obtain a bound for $|\nabla u_t|$ on $M \times [0, 1]$.

\vspace{4mm}

\section{Existence}

\vspace{4mm}

We shall use standard continuity method to prove Theorem \ref{Theorem1}. To start the continuity process, we need to construct an admissible function $w(x, t)$ which satisfies $w(x, 0) = u_0$ and $w(x, 1) = u_1$.
For this, we shall first establish Theorem \ref{Thm4}.

{\bf Proof of Theorem \ref{Thm4}.}
First, it is obvious to see that for $t = 0$, $u_0$ is the unique solution to \eqref{eq5-1}.
Also, the linearized operator with respect to the spacial variable $x$ associated to \eqref{eq5-1} is
invertible so that we can apply implicit function theorem to prove the openness of the set of $t \in [0, 1]$ at which \eqref{eq5-1} has an admissible solution $u(\cdot, t)$. The closedness can be established once we are able to derive $C^2$ estimates with respect to the spatial variable $x$. Then the existence of solution $u(x, t)$ to \eqref{eq5-1} for any $t \in [0, 1]$ can be obtained by continuity method, which can be further proved to be smooth with respect to $x$ by Evans-Krylov theory \cite{Evans, Krylov} and classical Schauder theory. In addition, we are sure that $u(\cdot, 1) = u_1$ by the uniqueness of solution to \eqref{eq5-1} for any $t \in [0, 1]$.

To give an upper bound for $u$,
assume that $u$ attains an interior maximum at $(x_0, t_0) \in M \times (0, 1)$. Then at $(x_0, t_0)$,
\[ \nabla u = 0, \quad \nabla^2 u \leq 0. \]
It follows that at $(x_0, t_0)$,
\[ \sigma_k \big( W[u] \big) \leq \sigma_k (A). \]
Consequently, at $(x_0, t_0)$,
\[ C^{- 1} \leq \psi =  e^{- 2 k u} \sigma_k \big( W[u] \big) \leq  e^{- 2 k u} \sigma_k (A) \leq C  e^{- 2 k u} . \]
We thus obtain an upper bound for $u(x_0, t_0)$ and consequently for $u$.

To give a lower bound for $u$,
assume that $u$ attains an interior minimum at $(x_1, t_1) \in M \times (0, 1)$. Then at $(x_1, t_1)$,
\[ \nabla u = 0, \quad \nabla^2 u \geq 0. \]
It follows that at $(x_1, t_1)$,
\[ \sigma_k \big( W[u] \big) \geq \sigma_k (A). \]
Consequently, at $(x_1, t_1)$,
\[ C \geq \psi =  e^{- 2 k u} \sigma_k \big( W[u] \big) \geq  e^{- 2 k u} \sigma_k (A) \geq C^{- 1}  e^{- 2 k u} . \]
We thus obtain a lower bound for $u(x_1, t_1)$ and consequently for $u$.

Recall that the global estimates for $|\nabla u|$ and $|\nabla^2 u|$ have been derived in Guan \cite{Guan08}. Hence we are able to obtain $C^2$ estimates with respect to the spatial variable $x$ for \eqref{eq5-1}.

In order to prove $u(x, t)$ to be smooth with respect to $(x, t)$, set
\[ w = \frac{u(x, t + \tau) - u(x, t)}{\tau}, \quad \tau \in \mathbb{R}. \]
Since
\[ \sigma_k \big( W[u] \big)(x, t) = e^{2 k u(x, t)} \psi(x, t)   \]
and
\[ \sigma_k \big( W[u] \big)(x, t + \tau) = e^{2 k u(x, t + \tau)} \psi(x, t + \tau) ,  \]
taking the difference and divided by $\tau$ yields
\begin{equation} \label{eq5-9}
\begin{aligned}
& \int_0^1 \sigma_k^{ij} \Big( (1 - \theta) W[u] (x, t) + \theta W[u] (x, t + \tau) \Big) d \theta \cdot \bigg( w_{ij} + \gamma \Delta w \delta_{ij} \\
& + s u_i (x, t + \tau) w_j + s u_j (x, t) w_i - \frac{r}{2} \Big( \nabla u (x, t + \tau) \cdot \nabla w + \nabla w \cdot \nabla u (x, t) \Big) \delta_{ij} \bigg) \\
= & 2 k w  \int_0^1 e^{2 k \big( (1 - \theta) u(x, t) + \theta u (x, t + \tau) \big)} \psi(x, t + \theta \tau) d \theta \\
& + \int_0^1 e^{2 k \big( (1 - \theta) u(x, t) + \theta u (x, t + \tau) \big)} \psi_t (x, t + \theta \tau) d \theta,
\end{aligned}
\end{equation}
which is a second order linear uniformly elliptic equation with respect to $x$. We may write it in the form
\[ \mathcal{T} (w) := a_{ij} w_{ij} + b_i w_i + c w = f. \]
To give an upper bound for $w$, we consider the test function
\[ \Phi_1 = w - c_1, \]
where $c_1$ is a positive constant to be chosen.
\[ \mathcal{T} ( \Phi_1 ) = f - c c_1. \]
Choosing $c_1$ sufficiently large depending on $\inf \psi$, we obtain
\[ \mathcal{T} ( \Phi_1 ) \geq 0, \]
which implies that
\[ \Phi_1 \leq 0 \quad \text{on} \quad  M \times [0, 1]. \]
We thus obtain an upper bound for $w$ on $M \times [0, 1]$.
To give a lower bound for $w$, we consider the test function
\[ \Phi_2 = w + c_2, \]
where $c_2 > 0$ is a sufficiently large constant depending on $\inf \psi$ such that
\[ \mathcal{T} ( \Phi_2 ) \leq 0. \]
We thus obtain a lower bound for $w$.
Now, since we have obtained a uniform bound (independent of $u$, independent of $\tau$, independent of $t$) for $w$, by Schauder interior estimate (see for instance \cite{GT}) we can infer that the set of functions $w$ and their first and second covariant derivatives $w_i$, $w_{ij}$ (i, j = 1, \ldots, n), are uniformly bounded and equicontinuous on $M$. Since $w \rightarrow u_t$ on $M$ as $\tau \rightarrow 0$, possibly passing to a subsequence, we may assert that $w_i \rightarrow u_{ti}$, $w_{ij} \rightarrow u_{tij}$ as $\tau \rightarrow 0$. Seeing that $u_t$ is in $C^2 (M)$ and satisfies
\begin{equation} \label{eq5-3}
 \sigma_k^{ij} \big( W[u] \big) \mathcal{M}_{ij} (u_t) = e^{2 k u} (2 k \psi u_t + \psi_t),
\end{equation}
by interior regularity theorem (see for instance \cite{GT}) we can assert that $u_t$ is smooth with respect to $x$.

Next, set
\[ w^{(1)} = \frac{u_t (x, t + \tau) - u_t (x, t)}{\tau}, \quad \tau \in \mathbb{R}, \]
\[ u^{\theta} = (1 - \theta) u(x, t) + \theta u(x, t + \tau), \quad 0 \leq \theta \leq 1, \]
\[ u_t^{\theta} = (1 - \theta) u_t (x, t) + \theta u_t (x, t + \tau). \]
Since
\[ \begin{aligned}
& \sigma_k^{ij} \big( W[u] \big) (x, t) \bigg( u_{tij} (x, t) + s u_i (x, t) u_{tj}(x, t) + s u_{ti}(x, t) u_j(x, t) + \\
& \Big( \gamma \Delta u_t(x, t) - r \nabla u (x, t) \cdot \nabla u_t (x, t) \Big) \delta_{ij} \bigg) = e^{2 k u (x, t)} \Big(2 k \psi (x, t) u_t (x, t) + \psi_t (x, t) \Big)
\end{aligned} \]
and
\[ \begin{aligned}
& \sigma_k^{ij} \big( W[u] \big) (x, t + \tau) \bigg( u_{tij} (x, t + \tau) + s u_i (x, t + \tau) u_{tj}(x, t + \tau) \\
& + s u_{ti}(x, t + \tau) u_j(x, t + \tau) + \Big( \gamma \Delta u_t (x, t + \tau) - r \nabla u (x, t + \tau) \cdot \nabla u_t (x, t + \tau) \Big) \delta_{ij} \bigg)  \\
& = e^{2 k u (x, t + \tau)} \Big(2 k \psi (x, t + \tau) u_t (x, t + \tau) + \psi_t (x, t + \tau) \Big),
\end{aligned} \]
taking the difference and divided by $\tau$ yields
\begin{equation} \label{eq5-10}
\begin{aligned}
& \int_0^1 \sigma_k^{ij} \Big( (1 - \theta) W[u] (x, t) + \theta W[u] (x, t + \tau) \Big) \bigg( w^{(1)}_{ij} + s u^{\theta}_i w^{(1)}_j \\
& + s w^{(1)}_i u^{\theta}_j + \Big( \gamma \Delta w^{(1)} - r \nabla u^{\theta} \cdot \nabla w^{(1)} \Big) \delta_{ij} + s w_i (u_t^{\theta})_j + s (u_t^{\theta})_i w_j \\
& - r \nabla w \cdot \nabla u_t^{\theta} \delta_{ij} \bigg) + \sigma_k^{ij, pq} \Big( (1 - \theta) W[u] (x, t) + \theta W[u] (x, t + \tau) \Big) \cdot \\
&  \bigg( (u_t^{\theta})_{ij} + s u^{\theta}_i (u_t^{\theta})_j + s (u_t^{\theta})_i u^{\theta}_j + \Big( \gamma \Delta u_t^{\theta} - r \nabla u^{\theta} \cdot \nabla u_t^{\theta} \Big) \delta_{ij}  \bigg) \cdot \bigg( w_{pq} + \gamma \Delta w \delta_{pq} \\
& + s u_p (x, t + \tau) w_q + s u_q (x, t) w_p - \frac{r}{2} \Big( \nabla u (x, t + \tau) \cdot \nabla w + \nabla w \cdot \nabla u (x, t) \Big) \delta_{pq} \bigg)  d \theta \\
= & \int_0^1 e^{2 k u^{\theta} } \bigg( 2 k \psi(x, t + \theta \tau) w^{(1)} + \psi_{tt} (x, t + \theta \tau) \\
& + 2 k \psi_t (x, t + \theta \tau)  u_t^{\theta} + 4 k^2 w  \psi (x, t + \theta \tau)  u_t^{\theta} + 2 k w \psi_t (x, t + \theta \tau)  \bigg) d \theta.
\end{aligned}
\end{equation}
We can see that the above equation is again a uniformly elliptic second order linear equation. Applying similar argument as above to \eqref{eq5-3}, we can obtain a uniform bound for $u_t$, which in turn implies a uniform bound for $|u_t|_{C^{2, \alpha}(M)}$ by Schauder interior estimate. Also, we have a uniform bound for $|w|_{C^{2, \alpha}(M)}$ by applying Schauder interior estimate to \eqref{eq5-9}. Applying similar argument as above to \eqref{eq5-10}, we obtain a uniform bound for $w^{(1)}$. Then we obtain a uniform bound for $|w^{(1)}|_{C^{2, \alpha}(M)}$ by applying Schauder interior estimate to \eqref{eq5-10}, which also implies that $w^{(1)}$, $w^{(1)}_i$, $w^{(1)}_{ij}$, $i, j = 1, \ldots, n$ are uniformly bounded and equicontinuous on $M$. Letting $\tau \rightarrow 0$, we can see that $w^{(1)} \rightarrow u_{tt}$. Possibly passing to a subsequence, $w^{(1)}_{i} \rightarrow u_{tti}$ and $w^{(1)}_{ij} \rightarrow u_{ttij}$. Since $u_{tt}$ is in $C^2 (M)$ and satisfies
\begin{equation*}
\begin{aligned}
& \sigma_k^{ij} \big( W[u] \big) \Big( \mathcal{M}_{ij} ( u_{tt} ) + 2 s u_{ti} u_{tj} - r |\nabla u_t|^2 \delta_{ij} \Big) + \sigma_k^{ij, pq} \big( W[u] \big) \mathcal{M}_{ij} ( u_t )  \mathcal{M}_{pq} ( u_t ) \\
= & e^{2 k  u } \Big( 2 k \psi u_{tt} + \psi_{tt} + 4 k^2 u_t^2  \psi + 4 k u_t \psi_t \Big) ,
\end{aligned}
\end{equation*}
$u_{tt}$ must be smooth with respect to $x$ by interior regularity theorem.

Higher order regularity with respect to $t$ follows from the same method as above.
\hfill \qedsymbol

\vspace{4mm}

{\bf Proof of Theorem \ref{Theorem1}.}
For the case when $\gamma > 0$,
by Theorem \ref{Thm4} we know that there exists a smooth solution $v(x, t)$ to
\begin{equation} \label{eq5-5}
\left\{ \begin{aligned}
e^{- 2 k u} \sigma_k \big( W[u] \big) = & (1 - t) e^{- 2 k u_0} \sigma_k \big( W[u_0] \big) + t e^{- 2 k u_1} \sigma_k \big( W[u_1] \big), \\
u(\cdot, 0) = & u_0,  \quad  u(\cdot, 1) = u_1
\end{aligned} \right.
\end{equation}
which satisfies $\lambda\big( W[ v ] \big)(x, t) \in \Gamma_k$ for any $(x, t) \in M \times [0, 1]$.

For the case when $\gamma = 0$ and $r > 0$, if $s = 0$, we can see that
\begin{equation*}
\begin{aligned}
& W \big[ (1 - t) u_0 + t u_1  \big] \\
= &   (1 - t) \nabla^2 u_0 + t \nabla^2 u_1  - \frac{r}{2} \big| (1 - t) \nabla u_0 + t \nabla u_1 \big|^2 g + A  \\
= & (1 - t) W [ u_0 ] + t W [ u_1 ] + \frac{r}{2} (1 - t) t \big| \nabla (u_0 - u_1) \big|^2 g.
\end{aligned}
\end{equation*}
Hence we can choose $v = (1 - t) u_0 + t u_1$ so that
$\lambda\big( W[ v ] \big)(x, t) \in \Gamma_k$ for any $(x, t) \in M \times [0, 1]$.
If $s \neq 0$, let
\[ v = \frac{1}{s} \ln \Big( (1 - t) e^{s u_0} + t e^{s u_1} \Big). \]
By direct calculation,
\begin{equation*}
\begin{aligned}
W [ v ] = & \nabla^2 v + s d v \otimes d v  - \frac{r}{2} |\nabla v|^2 g + A \\
= & \frac{(1 - t) \Big( (1 - t) e^{2 s u_0} + t e^{s (u_0 + u_1)} \Big) W [u_0] + t \Big( t e^{2 s u_1} + ( 1 - t ) e^{s (u_1 + u_0)} \Big) W [u_1]}{\Big( (1 - t) e^{s u_0} + t e^{s u_1} \Big)^2 }
\\
& + \frac{(1 - t) t e^{s (u_0 + u_1)} \frac{r}{2} \big| \nabla (u_0 - u_1) \big|^2 g}{\Big( (1 - t) e^{s u_0} + t e^{s u_1} \Big)^2}.
\end{aligned}
\end{equation*}
Hence $\lambda\big( W[ v ] \big)(x, t) \in \Gamma_k$ for any $(x, t) \in M \times [0, 1]$.

Let
\[ w(x, t) = v(x, t) + a t (t - 1).  \]
We may choose $a$ sufficiently large such that
\[ w_{tt} > 0 \quad \text{on} \quad M \times [0, 1] \]
and
\[ w_{tt} \sigma_k \big( W [v] \big) - \sigma_k^{ij} \big( W[ v ] \big) v_{ti} v_{tj} > 0 \quad \text{on} \quad M \times [0, 1]. \]
It follows that
\[ \lambda \big( E_w \big) \in \Gamma_k \quad \text{on} \quad M \times [0, 1].  \]

Now, we construct the continuity process for $\tau \in [0, 1]$,
\begin{equation} \label{eq5-6}
\left\{ \begin{aligned}
& u_{tt} \sigma_k \big( W[u] \big) - \sigma_k^{ij} \big( W[ u ] \big) u_{ti} u_{tj} \\
= & (1 - \tau) \Big( w_{tt} \sigma_k \big( W[w] \big) - \sigma_k^{ij} \big( W[ w ] \big) w_{ti} w_{tj}  \Big) + \tau \psi(x, t), \\
& u(\cdot, 0) = u_0,  \quad  u(\cdot, 1) = u_1.
\end{aligned} \right.
\end{equation}
It is obvious to see when $\tau = 0$, $u^0 = w$ is an admissible solution to \eqref{eq5-6}. Since the linearized operator associated to \eqref{eq5-6} is
invertible, we can apply implicit function theorem to prove the openness of the set of $\tau \in [0, 1]$ at which \eqref{eq5-6} has an admissible solution $u^{\tau}(x, t)$ on $M \times [0, 1]$. The closedness can be proved by the a priori estimates which are established in previous sections. Then the existence of solution $u^{\tau}(x, t)$ to \eqref{eq5-6} for any $\tau \in [0, 1]$ follows from classical continuity method. The uniqueness follows from maximum principle.
\hfill \qedsymbol

\vspace{2mm}

Before we give the proof of Theorem \ref{Thm-weak solution} and Theorem \ref{Theorem2}, we provide the definition of viscosity solution to
\eqref{eq1}
according to Definition 1.1 in \cite{LiYanyan2009}.
\begin{defn} \label{Def1}
Let $\Omega$ be an open subset of $M \times [0, 1]$. A continuous function $u$ in $\Omega$ is a viscosity supersolution of \eqref{eq1} if for any $(x_0, t_0) \in \Omega$ and $\varphi \in C^2(\Omega)$, if $u - \varphi$ has a local minimum at $(x_0, t_0)$, then $R_{\varphi} (x_0, t_0) \notin \overline{\mathcal{S}}$ or \begin{equation*}
\Big( \varphi_{tt} \sigma_k \big( W[ \varphi ] \big) -  \sigma_{k}^{ij} \big( W[ \varphi ] \big)  {\varphi}_{ti} {\varphi}_{tj} \Big) (x_0, t_0) \leq \psi(x_0, t_0) \quad \text{on} \quad M \times [0, 1],
\end{equation*}
where $R_{\varphi}$ is given in \eqref{eq4-11} and $\mathcal{S}$ is given in Proposition \ref{prop5}. A continuous function $u$ in $\Omega$ is a viscosity subsolution of \eqref{eq1} if for any $(x_0, t_0) \in \Omega$ and $\varphi \in C^2(\Omega)$, if $u - \varphi$ has a local maximum at $(x_0, t_0)$, then
\begin{equation*}
\Big( \varphi_{tt} \sigma_k \big( W[ \varphi ] \big) -  \sigma_{k}^{ij} \big( W[ \varphi ] \big)  {\varphi}_{ti} {\varphi}_{tj} \Big) (x_0, t_0) \geq \psi(x_0, t_0) \quad \text{on} \quad M \times [0, 1].
\end{equation*}
We say that $u$ is a viscosity solution of \eqref{eq1} if it is both a viscosity supersolution and a viscosity subsolution.
\end{defn}

\vspace{2mm}

{\bf Proof of Theorem \ref{Thm-weak solution}.}
We construct the following Dirichlet problem
\begin{equation} \label{eq5-8}
\left\{ \begin{aligned}
& u_{tt} \sigma_k \big( W^{\epsilon} [u] \big) - \sigma_k^{ij} \big( W^{\epsilon}[ u ] \big) u_{ti} u_{tj}
= \psi + \epsilon \quad \text{on  }  M \times [0, 1], \\
& u(\cdot, 0) = u_0,  \quad  u(\cdot, 1) = u_1,
\end{aligned} \right.
\end{equation}
where
\[ W^{\epsilon} [ u ] = W [ u ] + g^{- 1} \epsilon \Delta u g . \]
For any $\epsilon \in (0, 1]$, by Theorem \ref{Theorem1} there exists a unique smooth admissible solution $u^{\epsilon} (x, t)$ to \eqref{eq5-8}.
By \eqref{eq1-11}, we know that the solutions $\{ u^{\epsilon} \}$ have uniform $C^1$ bound which is independent of $\epsilon$.  Thus, as $\epsilon \rightarrow 0$, $u^{\epsilon}$ has a convergent subsequence which converges in $C^{0, \alpha}$ to a $C^{0, 1}$ solution $u$ of \eqref{eq1}--\eqref{eq1-4} for any $\alpha \in (0, 1)$. In the sense of Definition \ref{Def1}, $u$ is a viscosity solution.
\hfill \qedsymbol

\vspace{2mm}

{\bf Proof of Theorem \ref{Theorem2}.}
For any $\epsilon \in (0, 1]$, by Theorem \ref{Theorem1}, there exists a unique smooth admissible solution $u^{\epsilon} (x, t)$ to the Dirichlet problem
\begin{equation} \label{eq5-7}
\left\{ \begin{aligned}
& u_{tt} \sigma_k \big( W[u] \big) - \sigma_k^{ij} \big( W[ u ] \big) u_{ti} u_{tj}
= \epsilon \quad \text{on  }  M \times [0, 1], \\
& u(\cdot, 0) = u_0,  \quad  u(\cdot, 1) = u_1.
\end{aligned} \right.
\end{equation}
By the estimates established in previous sections, we know that the solutions $\{ u^{\epsilon} \}$ have uniform $C^2$ bound which is independent of $\epsilon$. Also, by the comparison principle, we know that $u^{\epsilon_1} \leq u^{\epsilon_2}$ if $\epsilon_1 \geq \epsilon_2$. Thus, as $\epsilon \rightarrow 0$, $u^{\epsilon}$ converges in $C^{1, \alpha}$ to a $C^{1, 1}$ solution $u$ of \eqref{eq1-1} for any $\alpha \in (0, 1)$. In the sense of Definition \ref{Def1}, $u$ is a viscosity solution of \eqref{eq1-1}.
\hfill \qedsymbol

\medskip

\vspace{4mm}

\section{Uniqueness}

\vspace{4mm}

In this section, we adopt the idea in Guan and Zhang \cite{Guan-Zhang2012} to establish the uniqueness result to the degenerate equation \eqref{eq1-1}.

\begin{lemma} \label{Lemma-approximation}
Suppose that (i) $\gamma > 0$,  $r \geq 0$ and $2 s k \leq r n$ or (ii) $r > 0$ and $2 s k \leq r n$.
Let $u$ be a $C^{1, 1}$ admissible function defined on $M \times [0, 1]$ which satisfies
\begin{equation*}
F_k (u) : = u_{tt} \sigma_k \big( W[u] \big) - \sigma_{k}^{ij} \big( W[u] \big)  u_{ti} u_{tj} = 0.
\end{equation*}
For any $\delta > 0$, there exists an admissible function $u_{\delta} \in C^{\infty} \big( M \times [0, 1] \big)$ such that
\[ \begin{aligned}
0 <  F_k (u_{\delta}) \leq \delta
\end{aligned} \]
and
\[ \Vert u - u_{\delta} \Vert_{C^0 (M \times [0, 1])} \leq \delta.  \]
\end{lemma}
\begin{proof}
We consider $v = (1 - \theta) u$, where $\theta \in (0, 1)$ is a constant to be chosen later.
By direct calculation,
\[ W [v] = (1 - \theta) W [u] + \theta \bigg( A + (1 - \theta) \Big( \frac{r}{2} |\nabla u|^2 g - s d u \otimes d u \Big) \bigg). \]
We note that if
\[ \Big( \frac{r}{2}, \ldots, \frac{r}{2}, \frac{r}{2} - s \Big) \in \overline{\Gamma}_k, \]
then
\[  \lambda \Big( \frac{r}{2} |\nabla u|^2 g - s d u \otimes d u \Big) \in \overline{\Gamma}_k . \]
It follows that
\[ \lambda \big( W [ v ] \big) \in \Gamma_k . \]
Meanwhile,
\[
\left(
\begin{array}{cc}
 u_{tt} & d u_t \\
 d u_t & W [ u ]
\end{array}
\right) \in \overline{S} \]
and
\[
\left(
\begin{array}{cc}
 0 & 0 \\
 0 &  A + (1 - \theta) \Big( \frac{r}{2} |\nabla u|^2 g - s d u \otimes d u \Big)
\end{array}
\right) \in \overline{S}  \]
imply that
\begin{footnotesize}
\[ \left(
\begin{array}{cc}
 v_{tt} & d v_t \\
 d v_t & W [ v ]
\end{array}
\right) = (1 - \theta) \left(
\begin{array}{cc}
 u_{tt} & d u_t \\
 d u_t & W [ u ]
\end{array}
\right) + \theta \left(
\begin{array}{cc}
 0 & 0 \\
 0 &  A + (1 - \theta) \Big( \frac{r}{2} |\nabla u|^2 g - s d u \otimes d u \Big)
\end{array}
\right)  \in \overline{S}. \]
\end{footnotesize}
Let
\[ w = v + \theta t (t - 1). \]
We can verify that
\[ \left(
\begin{array}{cc}
 w_{tt} & d w_t \\
 d w_t & W [ w ]
\end{array}
\right) \in S.  \]

Now for any $\delta > 0$, by continuity we can choose $\theta  \in (0, 1)$ sufficiently small  such that
\begin{footnotesize}
\[ \begin{aligned}
0 < & w_{tt} \sigma_k \big( W [ w ] \big) - \sigma_k^{ij} \big( W [ w ] \big) w_{ti} w_{tj}  \\
= &  \Big( (1 - \theta) u_{tt} + 2 \theta \Big) \sigma_k \Bigg(  (1 - \theta) W [u] + \theta \bigg( A + (1 - \theta) \Big( \frac{r}{2} |\nabla u|^2 g - s d u \otimes d u \Big) \bigg)  \Bigg) \\
& - \sigma_k^{ij} \Bigg( (1 - \theta) W [u] + \theta \bigg( A + (1 - \theta) \Big( \frac{r}{2} |\nabla u|^2 g - s d u \otimes d u \Big) \bigg) \Bigg) (1 - \theta)^2 u_{ti} u_{tj} \leq \frac{\delta}{2}
\end{aligned} \]
\end{footnotesize}
and
\[ |u - w| =  \theta \big| u - t (t - 1) \big| \leq \frac{\delta}{2}.  \]
We can then approximate $w$ by a smooth function $u_{\delta}$ such that
\[ \begin{aligned}
0 <  (u_{\delta})_{tt} \sigma_k \big( W [ u_{\delta} ] \big) - \sigma_k^{ij} \big( W [ u_{\delta} ] \big) (u_{\delta})_{ti} (u_{\delta})_{tj} \leq \delta
\end{aligned} \]
and
\[ |u - u_{\delta}| \leq \delta.  \]
\end{proof}

\begin{thm} \label{Theorem-uniqueness}
Under the assumption of Lemma \ref{Lemma-approximation}, $C^{1, 1}$ admissible solution to degenerate equation \eqref{eq1-1} is unique.
\end{thm}
\begin{proof}
Let $u_1$ and $u_2$ be two such solutions to \eqref{eq1-1}. For any $\delta > 0$, there exists an admissible function $v_1 \in C^{\infty} \big( M \times [0, 1] \big)$ such that
\[ \begin{aligned}
0 <  F_k (v_1) \leq \frac{\delta}{2}
\end{aligned} \]
and
\[ \Vert u_1 - v_1 \Vert_{C^0 (M \times [0, 1])} \leq \frac{\delta}{2}.  \]

For $\min\limits_{M \times [0, 1]} F_k (v_1) > 0$,
there exists an admissible function $v_2 \in C^{\infty} \big( M \times [0, 1] \big)$ such that
\[ \begin{aligned}
0 <  F_k (v_2) \leq \min\limits_{M \times [0, 1]} F_k (v_1) \leq F_k (v_1)
\end{aligned} \]
and
\[ \Vert u_2 - v_2 \Vert_{C^0 (M \times [0, 1])} \leq \min\limits_{M \times [0, 1]} F_k (v_1) \leq \frac{\delta}{2}.  \]

By the maximum principle, we know that
\[ \begin{aligned}
& \max\limits_{M \times [0, 1]} ( v_1 - v_2 ) \leq  \max\limits_{M \times \{ 0, 1 \}} ( v_1 - v_2 )  \\
\leq &   \max\limits_{M \times \{ 0, 1 \}} ( v_1 - u_1 ) + \max\limits_{M \times \{ 0, 1 \}} ( u_1 - u_2 ) +  \max\limits_{M \times \{ 0, 1 \}} ( u_2 - v_2 ) \leq \delta.
\end{aligned} \]
Hence we have
\[  \max\limits_{M \times [0, 1]} ( u_1 - u_2 ) \leq  \max\limits_{M \times [0, 1]} ( u_1 - v_1 ) +  \max\limits_{M \times [0, 1]} ( v_1 - v_2 ) +  \max\limits_{M \times [0, 1]} ( v_2 - u_2 ) \leq 2 \delta.  \]
Similarly, we can prove that
\[  \max\limits_{M \times [0, 1]} ( u_2 - u_1 ) \leq 2 \delta.  \]
Since $\delta > 0$ is arbitrary, letting $\delta \rightarrow 0$ we arrive at
\[ u_1 \equiv u_2 \quad \text{ on } M \times [0, 1]. \]
\end{proof}

\vspace{4mm}

\medskip
\noindent
{\bf Acknowledgements} \quad
The second author would like to express deep thanks to Wei Sun for bringing the idea of Theorem \ref{Thm4}. The second author is supported by National Natural Science Foundation of China (No. 12001138). The authors state that there is no conflict of interest. Data sharing not applicable to this article as no datasets were generated or analysed during the current study.

\end{document}